\documentclass[oneside,final]{amsart}
\usepackage{geometry}

\usepackage[utf8]{inputenc}
\usepackage{hyperref}
\usepackage[english]{babel}
\usepackage{enumerate}
\usepackage{geometry}
\usepackage{aliascnt}
\usepackage{amssymb}
\usepackage{amsmath}
\usepackage{amsthm}
\usepackage{verbatim}
\usepackage[normalem]{ulem}
\usepackage{mathtools}
\usepackage{esint} 
\usepackage{nth}
\usepackage{ragged2e}

\usepackage[capitalize]{cleveref}
\crefname{equation}{}{}
\usepackage[right]{showlabels}

\usepackage{tikz}
\usepackage{tikzscale}
\usepackage{graphicx}
\usepackage[labelformat=simple]{subcaption}
\usepackage{mathscinet} 

\usepackage[title]{appendix}

\usepackage{color}

\numberwithin{equation}{section} 

\usepackage{dsfont} 
\usepackage{hyperref} 

\usepackage{autonum} 

\allowdisplaybreaks[4] 

\theoremstyle{plain}

\newtheorem{proposition}{Proposition}[section]

\newaliascnt{lemma}{proposition} 

\newtheorem{lemma}[lemma]{Lemma}

\aliascntresetthe{lemma} 
\Crefname{lemma}{Lemma}{Lemmas}

\newaliascnt{theorem}{proposition} 
\newtheorem{theorem}[theorem]{Theorem}

\aliascntresetthe{theorem} 

\newaliascnt{corollary}{proposition} 
\newtheorem{corollary}[corollary]{Corollary}

\aliascntresetthe{corollary}

\newaliascnt{hypothesis}{proposition}

\aliascntresetthe{hypothesis} 

\theoremstyle{definition}

\newaliascnt{definition}{proposition} 
\newtheorem{definition}[definition]{Definition}

\aliascntresetthe{definition} 
\Crefname{definition}{Definition}{Definitions}

\newaliascnt{problem}{proposition} 

\aliascntresetthe{problem} 

\newaliascnt{example}{proposition} 

\aliascntresetthe{example} 

\theoremstyle{remark}

\newaliascnt{remark}{proposition} 
\newtheorem{remark}[remark]{Remark}

\aliascntresetthe{remark} 

\def\equationautorefname~#1\null{%
	(#1)\null
}
 
\newcommand{\R}{\mathbb{R}}
\newcommand{\C}{\mathbb{C}}

\newcommand{\N}{\mathbb{N}}
\newcommand{\E}{\mathcal{E}}

\newcommand{\curv}{\vec{\kappa}}
\newcommand{\scurv}{\kappa}
\renewcommand{\S}{\mathbb{S}}

\newcommand{\D}{\mathbb{D}}
\newcommand{\A}{\mathcal{A}}

\renewcommand{\H}{\mathbb{H}}
\newcommand{\W}{\mathcal{W}}

\newcommand{\defeq}{\vcentcolon=}
\newcommand{\Ll}{\mathcal{L}}

\newcommand{\cn}{\mathrm{cn}}

\newcommand{\dn}{\mathrm{dn}}

\newcommand*{\dd}{\mathop{}\!\mathrm{d}}

\def\nicefrac#1#2{%
    \raise.5ex\hbox{$#1$}%
    \kern-.15em/\kern-.05em%
    \lower.25ex\hbox{$#2$}}

\title[Singularities of the hyperbolic elastic flow]{Singularities of the hyperbolic elastic flow: Convergence, quantization and blow-ups}
\author[M.~Schlierf]{Manuel Schlierf}
\address[M.~Schlierf]{Institute for Applied Analysis, Ulm University, Helmholtzstraße 18, 89081 Ulm, Germany.}
\email{manuel.schlierf@uni-ulm.de}

\begin{document}
\begin{abstract}
  We study the elastic flow of closed curves and of open curves with clamped boundary conditions in the hyperbolic plane. While global existence and convergence toward critical points for initial data with sufficiently small energy is already known, this study pioneers an investigation into the flow's singular behavior. We prove a convergence theorem without assuming smallness of the initial energy, coupled with a quantification of potential singularities: Each singularity carries an energy cost of at least 8. Moreover, the blow-ups of the singularities are explicitly classified. \\A further contribution is an explicit understanding of the singular limit of the elastic flow of $\lambda$-figure-eights, a class of curves that previously served in showing sharpness of the energy threshold 16 for the smooth convergence of the elastic flow of closed curves.
\end{abstract}
\maketitle




\section{Introduction}

The study of minimizers of Euler's elastic energy is a captivating endeavor at the crossroads of mathematics and physics. This yields a better understanding of the equilibrium states that elastic structures naturally adopt under the influences of external forces, spanning from skyscrapers and bridges to biological tissues. A substantial amount of research has already been published on such minimizers and critical points. In particular, one approach focuses on the associated $L^2$-gradient flow of the elastic energy, the so-called \emph{elastic flow}.

In the following, we not only consider curves in Euclidean space but allow for other geometries in the ambient space. Let $(M,g)$ denote some Riemannian manifold with covariant derivative $\nabla$. Throughout this article, \emph{smooth} immersions $\gamma\colon I\to M$ are considered where $I$ is either an interval or $\S^1$. For the sake of brevity, we do not explicitly write ``smooth'' in the following but always assume immersions to be smooth if the regularity is not specified. Define $\partial_s = \frac{1}{|\partial_x\gamma|_g}\partial_x\ $. Then $\curv=\nabla_{\partial_s\gamma}\partial_s\gamma$ denotes its curvature and
\begin{equation}
    \E(\gamma) = \int_I |\curv|_g^2\dd s
\end{equation}
the elastic energy of $\gamma$ where $\dd s = |\partial_x\gamma|_g\dd x$ is the arc-length element. Its $L^2(\dd s)$-gradient is denoted by $\nabla\E$ and determines the speed for the elastic flow. Furthermore, we usually denote the length of $\gamma$ in $M$ by $\Ll_M(\gamma)=\int_I\dd s$. More details are given in \cref{sec:geo-prelim}.

\subsection{Overview: State of the art}

For the elastic flow of closed curves in Euclidean space $\R^n$, one cannot expect convergence due to scaling properties: starting from any round circle in $\R^n$, the Euclidean elastic flow exists globally in time and simply scales up the circle with diameter diverging to $\infty$. As circles can have arbitrarily small Euclidean elastic energy, there is no energy bound below which one can expect convergence of solutions to a critical point for $t\to\infty$. However, if one introduces constraints on the length --- either via a Lagrange multiplier resulting in fixed length or by modifying the energy functional so that the length is penalized, then convergence is obtained for $t\to\infty$. All these observations for closed curves are given in \cite{dziukkuwertschaetzle2002}. In particular, long-time existence and convergence follows by parabolic interpolation techniques. In \cite{lin2012}, these methods are generalized to the setting of open curves with Dirichlet boundary conditions, so-called \emph{clamped} conditions. Short-time existence for initial data in the energy space $W^{2,2}([0,1],\R^n)$ and global existence and convergence for the length-preserving clamped elastic flow is shown in \cite{ruppspener2020}.

As it turns out, the properties of the elastic flow heavily depend on the geometry of the ambient space. If the ambient is given by the two-sphere, i.e. a two-manifold with constant sectional curvature $1$, the analogous result to the Euclidean case is proven in \cite{dallacqualauxlinpozzispener2018}. That is, one has global existence in time, but convergence to a critical point only if the length is penalized. However, it is still an open question whether length-penalization is actually necessary in this ambient manifold. 

Contrarily, this question is partially answered for the elastic flow in the ambient space given by the two-manifold with constant sectional curvature equal to $-1$, the so-called \emph{hyperbolic plane} $\H^2$. Long-time existence is proved in \cite{dallacquaspener2017} with similar interpolation estimates as in the other ambient cases. However, in \cite{muellerspener2020}, the authors prove a Reilly-type inequality which allows to bound the hyperbolic length of a curve by its hyperbolic elastic energy, provided the elastic energy remains uniformly below $16$ which enables proving convergence of the elastic flow if the initial energy is below $16$. The authors further show that this energy threshold is sharp. More precisely, a sequence of initial data is constructed --- so-called $\lambda$-figure-eights --- whose elastic energy approaches $16$ from above and for which the hyperbolic length blows up along the elastic flow for $t\to\infty$.

While the qualitative behavior of the elastic flow in the hyperbolic plane seems quite similar to that of the Willmore flow of spherical immersions in $\R^3$, i.e. one has global existence and convergence to a critical point below a certain, sharp energy threshold, cf. \cite{kuwertschaetzle2002,kuwertschaetzle2001,kuwertschaetzle2004}, there is also a striking quantitative relation of these geometric flows. More details and some new insights gained in this article outlining the differences of both flows are outlined later on.

Actually, whenever mentioning ``convergence'' of the elastic flow to a critical point, the above contributions mostly only prove ``sub-convergence''. By proving a \L ojasiewicz-Simon gradient inequality for a wide range of ambient manifolds, by extending the arguments in \cite{dallacquapozzispener2016} to the setting of general ambient manifolds, \cite{pozzetta2022} promotes this to full convergence for the elastic flow of closed curves.

It is directly noticeable that in the above overview no results appear in the direction of an exact description of the singular behavior of the elastic flow in $\H^2$. Indeed, while \cite{muellerspener2020} construct initial data whose elastic flow develops a singularity, they cannot say any more than that hyperbolic lengths blow up for large times. Clearly, in this context one can no longer expect smooth convergence towards critical points --- one needs suitable weaker types of convergence. For this we refer to \cite{choksiveneroni2013,choksimorandottiveneroni2013} where the authors investigate the existence of minimizers of the Canham-Helfrich energy in the class of rotationally symmetric surfaces. For this purpose, they also work at the level of profile curves and use a suitable weak notion of regularity with appropriate parametrization to obtain compactness.

\subsection{Main results}

To state the main results, we consider the Poincar\'e model $\D^2=\{(p^{(1)},p^{(2)})^t\in\R^2:(p^{(1)})^2+(p^{(2)})^2<1\}$ for the hyperbolic plane. $\D^2$ is a two-manifold with Riemannian metric given by $g_p = \frac{4}{(1-|p|^2)^2}\langle\cdot,\cdot\rangle$ where $\langle\cdot,\cdot\rangle$ is the Euclidean scalar product in $\R^2$. In the following, $\overline{\D^2}=\{p\in\R^2:|p|\leq 1\}$ denotes the \emph{Euclidean} closure of $\D^2$ in $\R^2$. Furthermore, a curve $\gamma$ is said to be parametrized by constant Euclidean speed if $|\partial_x\gamma|\equiv\mathrm{const}$.

We first prove a convergence result for the elastic flow in the hyperbolic plane --- without any assumptions on the elastic energy of the initial datum. 

\begin{theorem}\label{thm:quan-clo}
    Let $\gamma_0\colon\S^1\to\D^2$ be an immersion and $\gamma\colon[0,\infty)\times\S^1\to\D^2$ a smooth family of immersions with 
    \begin{equation}\label{eq:clo-ef}
        \begin{cases}
            \partial_t\gamma = -\nabla\E(\gamma)&\text{in $[0,\infty)\times \S^1$}\\
            \gamma(0)=\gamma_0 &\text{in $\S^1$}.
        \end{cases}
    \end{equation}
    For any sequence of times $t_n\nearrow \infty$, there are isometries $\phi_n\colon\D^2\to\D^2$ such that, for $\gamma_n=\phi_n\circ\gamma(t_n)\colon\S^1\to\D^2$ and reparametrizations $\widetilde{\gamma_n}$ by constant Euclidean speed, there exists $\widetilde{\gamma_{\infty}}\in C^0(\S^1,\overline{\D^2})\cap C^{\infty}(\S^1\setminus \{x_1,\dots,x_m\},\D^2)$ for some $x_1,\dots,x_m\in\S^1$ and $m\in\N_0$ such that $\widetilde{\gamma_n}\to\widetilde{\gamma_{\infty}}$ uniformly and in $C^{\infty}_{\mathrm{loc}}(\S^1\setminus\{x_1,\dots,x_m\})$, after passing to a subsequence.

    Furthermore, $\{|\widetilde{\gamma_{\infty}}|=1\}=\{x_1,\dots,x_m\}$, $\nabla\E(\widetilde{\gamma_{\infty}})=0$ on $\S^1\setminus\{x_1,\dots,x_m\}$ and
    \begin{equation}\label{eq:en-ie-clo}
        \E(\widetilde{\gamma_{\infty}}|_{\S^1\setminus\{x_1,\dots,x_m\}}) \leq \E(\gamma_0) - 8\cdot m.
    \end{equation}
    Lastly, if $m\geq 1$, writing $\S^1\setminus\{x_1,\dots,x_m\} = \bigcup_{i=1}^m I_i$ for disjoint open intervals $I_i$, $\widetilde{\gamma_{\infty}}|_{I_i}$ either parametrizes a geodesic or an asymptotically geodesic elastica. 
\end{theorem}

Here we make use of the classification of hyperbolic elastica, i.e. critical points of $\E$, given in \cite{langersinger1984}, cf. \Cref{fig:elastica}. More details are revised later on.

\begin{figure}[!ht]
    \centering
    \begin{subfigure}{0.46\textwidth}
        \includegraphics[width=\linewidth]{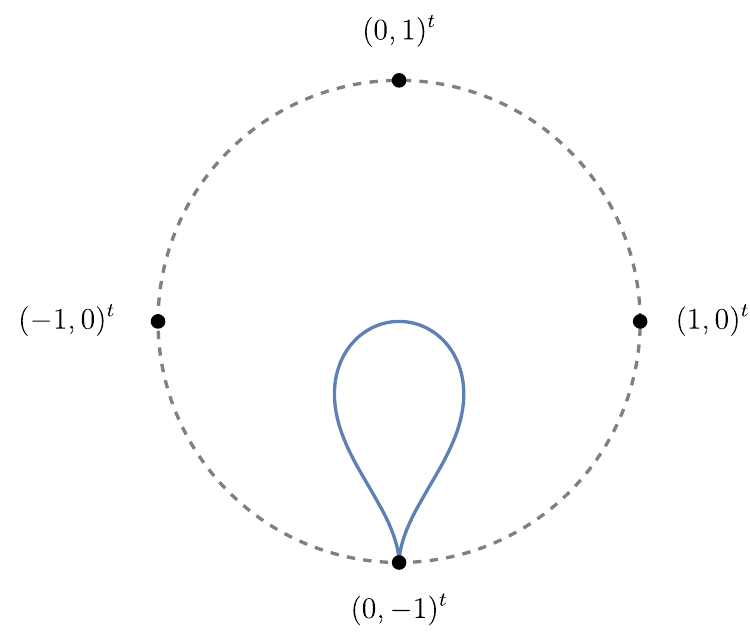}
    \end{subfigure}
    \qquad
    \begin{subfigure}{0.46\textwidth}
        \includegraphics[width=\linewidth]{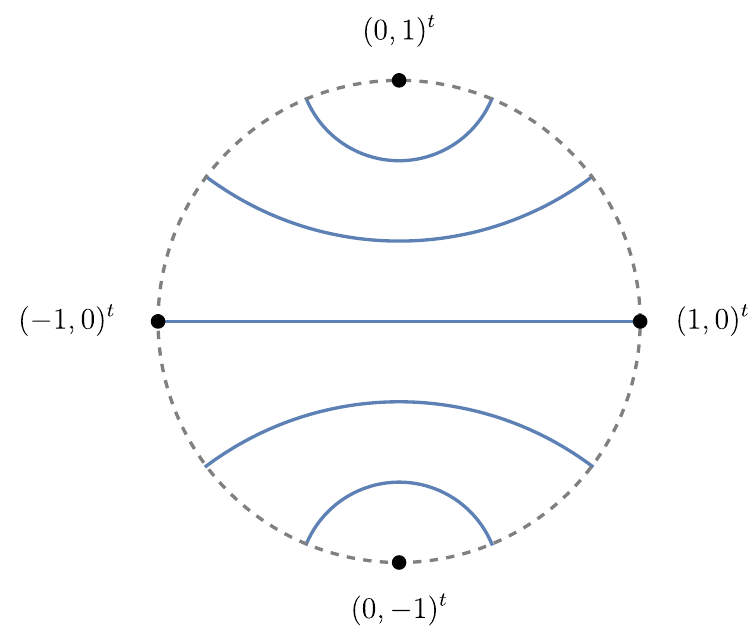}
    \end{subfigure}
    \caption{Plots of an asymptotically geodesic elastica (left) and geodesics (right) in $\D^2$.}
    \label{fig:elastica}
\end{figure}

The first observation one draws from \cref{thm:quan-clo} is that singularities only form at the ``boundary'' of $\D^2$ in $\R^2$, i.e. at ``hyperbolic infinity''. Moreover, the elastic energy of the initial immersion provides an upper bound on the number of singularities that can form at the boundary of $\D^2$ --- namely, by \cref{eq:en-ie-clo}, any such singularity costs at least $8$ units in terms of the elastic energy. 

In fact, as the proof of \cref{thm:quan-clo} reveals, there are only two possibilities one can observe. Either $\gamma(t_n)$ converges uniformly to a point on $\partial\D^2=\S^1\subseteq\R^2$. Or, if this is not the case, the isometries $\phi_n$ in \cref{thm:quan-clo} can be chosen trivially, i.e. $\phi_n(z)=z$. The first case can be seen as an ``implosion'' of the flow down to a single point on the boundary of $\D^2$. However, one then obtains a ``blow-up'' analysis of the imploding flow and states that the singular blow-up limit still only has finitely many singularities and is given as geodesics or asymptotically geodesic elastica everywhere else.
 
We later consider a special class of initial data $\gamma_0$ (which we call symmetric figure-eights) leading to singularities in the elastic flow for which there are no implosions in the above sense. Even more, $\widetilde{\gamma_{\infty}}$ does not depend on the choice of sequence $(t_n)_{n\in\N}$ for this class.
 
The estimate in \cref{eq:en-ie-clo} corresponds to typical energy inequalities in the literature relating the energies of singular limits of flows with those of the initial datum in terms of the number $m$ of singularities in the singular limit, cf. \cite[Theorem 6.6]{struwe2008}.

We also show the corresponding version of \cref{thm:quan-clo} for the \emph{clamped} elastic flow. 

\begin{theorem}\label{thm:quan-cla}
    Let $\gamma_0\colon[-1,1]\to\D^2$ be an immersion and $\gamma\colon[0,\infty)\times[-1,1]\to\D^2$ a smooth family of immersions evolving by elastic flow with clamped boundary conditions, that is, satisfying
    \begin{equation}\label{eq:cla-ef}
        \begin{cases}
            \partial_t\gamma = -\nabla\E(\gamma)&\text{in $[0,\infty)\times [-1,1]$}\\
            \gamma(0)=\gamma_0 &\text{in $[-1,1]$}\\
            \gamma(t,y) = \gamma_0(y)&\text{for $t\in[0,\infty)$ and $y=\pm1$}\\
            \partial_s\gamma(t,y) = \partial_s\gamma_0(y)&\text{for $t\in[0,\infty)$ and $y=\pm1$}.
        \end{cases}
    \end{equation}
    For any sequence of times $t_n\nearrow \infty$, writing $\widetilde{\gamma_n}$ for the reparametrization of $\gamma(t_n)$ by constant Euclidean speed on $[-1,1]$ for each $n\in\N$, there exists $\widetilde{\gamma_{\infty}}\in C^0([-1,1],\overline{\D^2})\cap C^{\infty}([-1,1]\setminus \{x_1,\dots,x_m\},\D^2)$ for some $x_1,\dots,x_m\in(-1,1)$ and $m\in\N_0$ such that $\widetilde{\gamma_n}\to\widetilde{\gamma_{\infty}}$ uniformly and in $C^{\infty}_{\mathrm{loc}}([-1,1]\setminus\{x_1,\dots,x_m\})$, after passing to a subsequence.

    Furthermore, $\{|\widetilde{\gamma_{\infty}}|=1\}=\{x_1,\dots,x_m\}$, $\nabla\E(\widetilde{\gamma_{\infty}})=0$ on $[-1,1]\setminus\{x_1,\dots,x_m\}$ and 
    \begin{equation}\label{eq:en-ie-cla}
        \E(\widetilde{\gamma_{\infty}}|_{[-1,1]\setminus\{x_1,\dots,x_m\}}) \leq \E(\gamma_0) - 8\cdot m.
    \end{equation}
    Lastly if $m\geq 1$, $\widetilde{\gamma_{\infty}}|_{(x_i,x_{i+1})}$ is either a geodesic or an asymptotically geodesic elastica, for all $i=0,\dots,m$ where $x_0=-1$ and $x_{m+1}=1$.
\end{theorem}

The main distinguishing feature is that, due to the Dirichlet boundary conditions, ``implosions'' cannot occur. Thus, one does not require the ``blow-up-isometries'' $\phi_n$ for the clamped flow. So, the singular limit $\widetilde{\gamma_{\infty}}$ satisfies the boundary data of the flow, i.e.
\begin{equation}
    \widetilde{\gamma_{\infty}}(\pm1) = \gamma_0(\pm1) \quad\text{and}\quad\partial_s\widetilde{\gamma_{\infty}}(\pm1)=\partial_s\gamma_0(\pm1).
\end{equation}

\begin{remark}
    We give an energy threshold for $\E(\gamma_0)$ below which \cref{thm:quan-cla} yields full convergence of suitable reparametrizations of $\gamma(t)$ for $t\to\infty$. For $p_{\pm1}\in\D^2$ and $v_{\pm1}\in T_{p_{\pm1}}\D^2$ with $|v_{\pm1}|_g=1$, define 
    \begin{align}
        M(p_{\pm1},v_{\pm1}) = \inf\big\{\E(\gamma) &\mid\gamma\in C^0([-1,1],\overline{\D^2})\cap C^{\infty}([-1,1]\setminus\{x_0\},\D^2)\text{ for }x_0\in (-1,1)\\
        &\quad\text{with }|\gamma(x_0)|=1, \gamma(\pm1)=p_{\pm1}\text{ and } \partial_s\gamma(\pm1)=v_{\pm1}\\
        &\quad\text{such that } \gamma|_{(-1,x_0)},\gamma|_{(x_0,1)}\text{ either is a segment of a geodesic} \\
        &\quad \text{or an asymptotically geodesic elastica, respectively}\big\}.
    \end{align}
    With this definition, if 
    \begin{equation}
        \E(\gamma_0)<8+M\big(\gamma_0(\pm1),\partial_s\gamma_0(\pm1)\big),
    \end{equation}
    then \eqref{eq:en-ie-cla} yields $m=0$ and therefore $\Ll_{\D^2}(\widetilde{\gamma_{\infty}})<\infty$. This is due to the properties of $\widetilde{\gamma_{\infty}}$ in \cref{thm:quan-cla}. In particular, a standard subsequence argument yields that the hyperbolic length is uniformly bounded along the evolution, that is, $\sup_{t\in[0,\infty)}\Ll_{\D^2}(\gamma(t))<\infty$. Then full convergence of suitable reparametrizations of $\gamma(t)$ for $t\to\infty$ follows from \cite[Theorem~4.5 with $a\equiv -2$]{schlierf2023}.

    Note that this energy threshold is equivalent to the recently introduced energy threshold in \cite[Theorem~1.1]{eichmann2024} (and independently obtained in \cite[Theorem~1.3 and Appendix~D]{schlierf2024DirichletWillmore}) for the convergence of the clamped Willmore flow of surfaces of revolution.
\end{remark}

Lastly, blow-ups at those singularities $x_j$ always yield asymptotically geodesic elastica: 

\begin{proposition}[Blow-up analysis]\label{prop:bu-clo}
    In the setting of \cref{thm:quan-clo} or \cref{thm:quan-cla}, fix $1\leq j\leq m$. Then, for any $\delta>0$ sufficiently small, there are isometries $\psi_n\colon\D^2\to\D^2$ (``blow-ups'') such that, for the reparametrizations 
    \begin{equation}
        \widetilde{\tau_n}\colon[-1,1]\to\D^2\text{ of }\tau_n=\psi_n\circ\widetilde{\gamma_n}|_{(x_j-\delta,x_j+\delta)}\text{ by constant Euclidean speed},
    \end{equation}
    there are $\widetilde{\tau_{\infty}}\colon[-1,1]\to\overline{\D^2}$, $n_j\in\N$ and $y_1<\dots<y_{n_j-1}\in(-1,1)$ such that $\widetilde{\tau_{\infty}}|_{(y_{i-1},y_{i})}$ all parametrize an asymptotically geodesic elastica with constant Euclidean speed, respectively for $i=1,\dots,n_j$ where $y_0=-1$ and $y_{n_j}=1$. Furthermore, $\widetilde{\tau_n}\to\widetilde{\tau_{\infty}}$ uniformly in $[-1,1]$ and smoothly on compact intervals in $(-1,1)\setminus\{y_1,\dots,y_{n_j-1}\}$.

    Lastly, one has the improved energy inequality
    \begin{equation}\label{eq:impr-ei}
        \E(\widetilde{\gamma_{\infty}}|_{I\setminus\{x_1,\dots,x_m\}})\leq \E(\gamma_0) - 8\sum_{j=1}^m n_j.
    \end{equation}    
\end{proposition}

For two special classes of initial data, one for the elastic flow of closed curves, the other for the clamped elastic flow, we can show even stronger results. In fact, for the reparametrizations by constant Euclidean speed, we can show convergence for $t\to\infty$ to an explicitly given geodesic. In particular, the limit does not depend on a subsequence of times anymore. Furthermore, for the flow of closed curves with this initial datum, we can exclude aforementioned ``implosions'' so that the isometries $\phi_n$ are not needed. In the following theorem, one makes the identification $\S^1=\nicefrac{[-2,2]\ }{\sim}$ where the endpoints $\pm2$ are equivalent by $\sim$.

\begin{theorem}\label{thm:sf8-conv}
    Let $\gamma_0\colon\S^1\to\D^2$ be a symmetric figure-eight with $16<\E(\gamma_0)<32$ and denote by $\gamma\colon[0,\infty)\times\S^1\to\D^2$ a smooth family of immersions satisfying \cref{eq:clo-ef}. Then there are reparametrizations $\widetilde{\gamma(t)}$ of $\gamma(t)$ with constant Euclidean speed such that $\widetilde{\gamma(t)}\to \widetilde{\gamma_{\infty}}$ uniformly for $t\to\infty$ and smoothly on $\S^1\setminus\{\pm 1\}$ where $\widetilde{\gamma_{\infty}}\colon\S^1\to\overline{\D^2}$,
    \begin{equation}
        \widetilde{\gamma_{\infty}}(x)=\begin{cases}
            (0,-x-2)^t&-2\leq x\leq-1\\
            (0,x)^t&-1\leq x\leq 1\\
            (0,-x+2)^t&1\leq x\leq 2.
        \end{cases}
    \end{equation}
\end{theorem}

More details on the definition of ``symmetric figure-eights'' are provided in \cref{sec:fig-eights}. Here it suffices to note that the aforementioned examples of initial data resulting in singular behavior for the elastic flow in \cite{muellerspener2020} is contained in the class of such symmetric figure-eights. They are called $\lambda$-figure-eights and are visualized in \Cref{fig:lamb-fig-eights}.

\begin{figure}[!ht]
    \centering
    \begin{subfigure}{0.46\textwidth}
        \includegraphics[width=\linewidth]{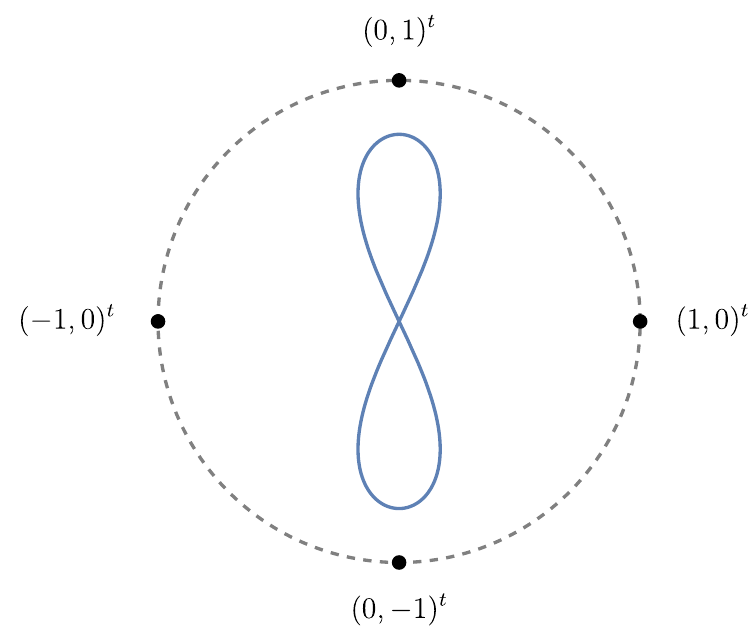}
    \end{subfigure}
    \qquad
    \begin{subfigure}{0.46\textwidth}
        \includegraphics[width=\linewidth]{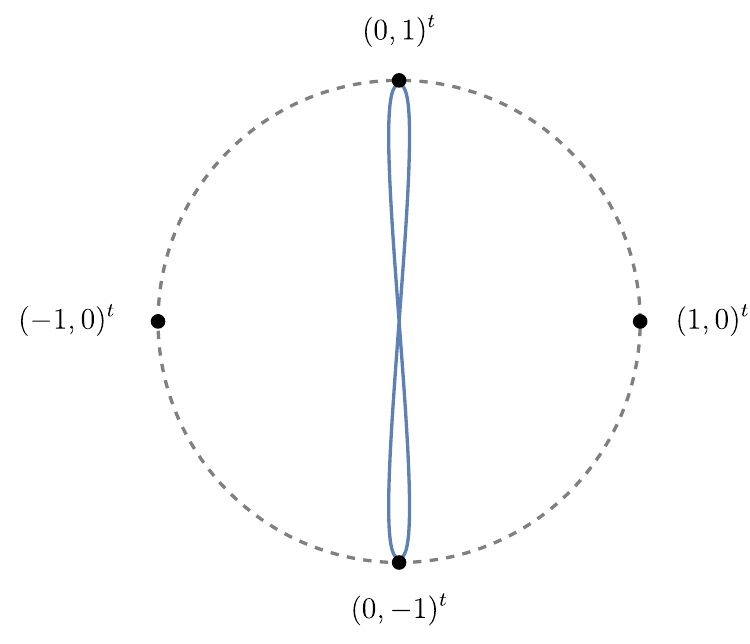}
    \end{subfigure}
    \caption{Plots of $\lambda$-figure-eights for $\lambda=\frac12$ (left) and $\lambda=\frac{1}{200}$ (right).}
    \label{fig:lamb-fig-eights}
\end{figure}

A discussion for a corresponding class of curves for the clamped flow containing the initial data constructed in \cite[Section 5]{schlierf2023} is provided in \cref{sec:vert-cla}.

\subsection{Willmore flow: Quantitative relations and qualitative deviations}

Another, isometric description for the hyperbolic plane is given by the half-plane model $\H^2=\{(u^{(1)},u^{(2)})^t\in\R^2:u^{(2)}>0\}$ with metric $g_u=\langle\cdot,\cdot\rangle_g=\frac{1}{(u^{(2)})^2}\langle\cdot,\cdot\rangle$. Let $u\colon I\to\H^2$ be an immersion where $I$ is either an interval or $\S^1$. The surface of revolution $f_u$ canonically associated to the profile curve $u$ is defined on $\Sigma=I\times\S^1$ by $f_u\colon\Sigma\to\R^3$ where 
\begin{equation}\label{eq:def-surf-rev}
    f(x,\varphi)= (u^{(1)}(x),u^{(2)}(x)\cos(\varphi),u^{(2)}(x)\sin(\varphi))^t.
\end{equation}
Consider an immersion $f\colon \Sigma'\to \R^3$ of an orientable surface $\Sigma'$ with not necessarily empty boundary. Then $H$ denotes the mean curvature, i.e. the arithmetic mean of the principal curvatures of $f$, and $\mu_f$ the area-element induced by the metric $f^*\langle\cdot,\cdot\rangle$ on $\Sigma'$. Its \emph{Willmore energy} is defined by $\W(f)=\int_{\Sigma'} H^2\dd \mu_f$. The following fundamental relation between $\W(f_u)$ and $\E(u)$ goes back to \cite{bryantgriffiths1986}. In its cited form, it can be found in \cite[Equation (2.11)]{eichmanngrunau2019}. Writing $|\cdot|$ for the Euclidean norm in $\R^2$, one has
\begin{equation}\label{eq:br-gr}
    \frac{2}{\pi}\W(f_u) = \E(u) - 4\frac{\partial_xu^{(2)}}{|\partial_xu|}\Big|_{\partial I}.
\end{equation}

While \cref{eq:br-gr} already quantifies the relation between the hyperbolic elastic energy of a profile curve and the Willmore energy of its surface of revolution, there also is a direct relation between the associated $L^2$ gradients --- and thus between the elastic flow and the famous Willmore flow. More precisely, by computations in \cite[Theorem 4.1]{dallacquaspener2018}, one has that, if $u\colon [0,T)\times I\to\H^2$ satisfies
\begin{equation}\label{eq:wf-vs-ef}
    \partial_t u = -\frac{1}{4(u^{(2)})^4} \nabla\E(u),
\end{equation}
then the associated surfaces of revolution satisfy the Willmore flow equation
\begin{equation}
    \partial_t f_u = - \nabla\W(f_u).
\end{equation}
So solving the Willmore flow starting in a surface of revolution equates solving \cref{eq:wf-vs-ef} which obviously is strongly related to solving the elastic flow.

Consider now the aforementioned $\lambda$-figure-eights constructed in \cite{muellerspener2020}. These are used in \cite{dallacquamullerschatzlespener2020} as initial data for the Willmore flow \cref{eq:wf-vs-ef}. The authors then show that the Willmore flow starting in those figure-eights also develops singularities.

Now our result in \cref{thm:sf8-conv} translated to the half-plane model $\H^2$ yields that the solution to the elastic flow starting in one such symmetric figure-eight converges (in a suitable sense) to a double-parametrization of the vertical geodesic with trace $\{(0,y)^t:y>0\}$. In terms of the associated surfaces of revolution, one thus gets convergence to a plane (in a suitable sense), cf. \Cref{fig:wf-vs-ef}.

However, numerical simulations in \cite[Appendix~B, Figures~B.2 -- B.4]{barrettgarckenuernberg2021} suggest that the solution to \eqref{eq:wf-vs-ef}, i.e. the Willmore flow, asymptotically resembles the left illustration in \Cref{fig:wf-vs-ef}. So the Willmore flow stays bounded in Euclidean space, while the elastic flow becomes unbounded and loses all of its Willmore energy in the limit! Furthermore, for the elastic flow, global existence is always known while, for the Willmore flow, it is only conjectured \cite[p.~883]{barrettgarckenuernberg2021} but generally not known.

\begin{figure}[h!]
    \centering
    \begin{subfigure}{0.45\textwidth}
        \includegraphics[width=\linewidth]{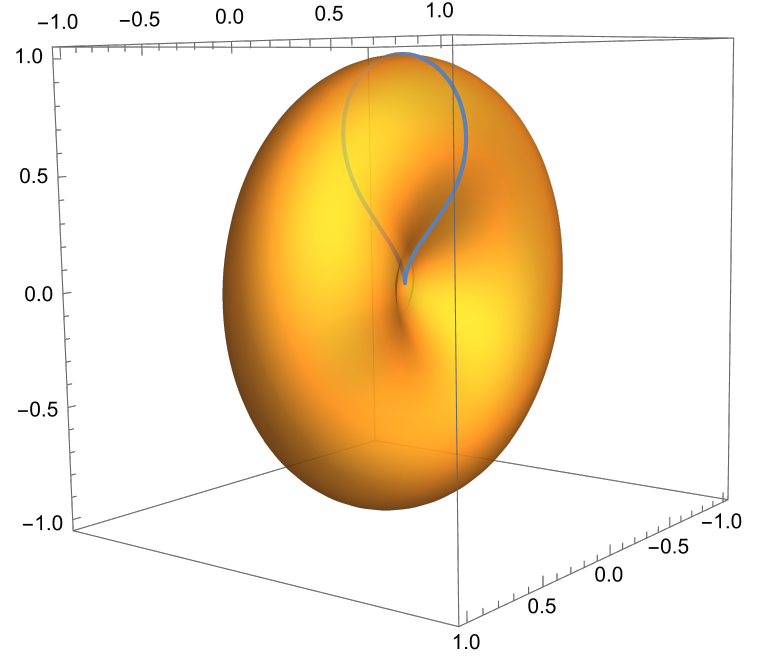}
    \end{subfigure}
    \qquad
    \begin{subfigure}{0.45\textwidth}
        \includegraphics[width=\linewidth]{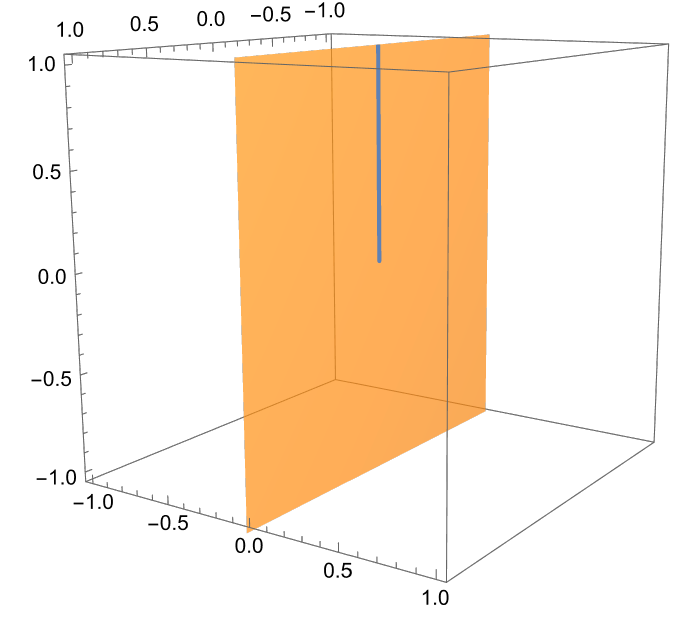}
    \end{subfigure}
    \caption{Surfaces of revolution one obtains as limits of Willmore flow (left) and elastic flow (right).}
    \label{fig:wf-vs-ef}
\end{figure}

\subsection{Outline and structure of the proofs}

\cref{sec:geo-prelim} first reviews some basic geometric preliminaries on the nature of the isometry between $\H^2$ and $\D^2$ as well as some fundamental relations on curvature and the elastic energy in the hyperbolic plane. Then, the fundamental estimates for proving the main results are developed. Namely, in \cref{subsec:en-est}, a general quantification result is proved in $\H^2$ which is the crucial step in showing that $\{|\widetilde{\gamma_{\infty}}|=1\}$, i.e. the set of singularities, is finite in \cref{thm:quan-clo,thm:quan-cla}. Furthermore, in \cref{subsec:len-est}, a relation between approaching ``hyperbolic infinity'', i.e. $\partial \D^2\subseteq\R^2$, the Euclidean length and the hyperbolic elastic energy is made. This turns out to be crucial in lifting the quantification result to the Poincar\'e disk $\D^2$.

With these preliminaries, \cref{sec:proof-main} is devoted to proving \cref{thm:quan-clo,thm:quan-cla,prop:bu-clo}. However, in the interest of brevity, we first provide some meta-results on compactness and on boundedness of the Euclidean length for solutions to \cref{eq:clo-ef,eq:cla-ef}.

In \cref{sec:fig-eights}, \cref{thm:sf8-conv} is proved while a corresponding construction and analysis for the clamped elastic flow is carried out in \cref{sec:vert-cla}. 

\section{Geometric preliminaries}\label{sec:geo-prelim}

\subsection{Curves in the hyperbolic plane: half-plane vs. Poincar\'e model}

Throughout this article, two isometric descriptions of the hyperbolic plane are considered. Recall that, by $\H^2=\{(u^{(1)},u^{(2)})^t\in\R^2:u^{(2)}>0\}$ we denote the \emph{half-plane} model and by $\D^2=\{(p^{(1)},p^{(2)})^t\in\R^2:(p^{(1)})^2+(p^{(2)})^2<1\}$ the \emph{Poincar\'e} model. 

In the half-plane model, the metric $g$ at a point $u=(u^{(1)},u^{(2)})^t\in\H^2$ is given by $\langle\cdot,\cdot\rangle_g=\frac{1}{(u^{(2)})^2}\langle\cdot,\cdot\rangle$. In the Poincar\'e model, the metric $g$ at a point $p\in\D^2$ is defined by $\langle\cdot,\cdot\rangle_g = \frac{4}{(1-|p|^2)^2}\langle\cdot,\cdot\rangle$. Here, $|\cdot|$ and $\langle\cdot,\cdot\rangle$ refer to the Euclidean norm and scalar product in $\R^2$, respectively. An isometry $\Phi\colon\D^2\to\H^2$ is given by 
\begin{equation}\label{eq:def-Phi}
    \Phi(p^{(1)},p^{(2)}) = \frac{1}{(p^{(1)})^2+(p^{(2)}-1)^2} \big(2p^{(1)},1-(p^{(1)})^2-(p^{(2)})^2\big)^t
\end{equation}
while its inverse satisfies 
\begin{equation}\label{eq:inv-Phi}
    \Phi^{-1}(u^{(1)},u^{(2)}) = \frac{1}{(u^{(1)})^2+(u^{(2)}+1)^2} \big( 2u^{(1)}, (u^{(1)})^2+(u^{(2)})^2-1 \big)^t.
\end{equation}
In particular, for its differential $D\Phi(p^{(1)},p^{(2)})$, one gets the expression
\begin{equation}\label{eq:Jacobi-Phi}
    \frac{2}{((p^{(1)})^2+(p^{(2)}-1)^2)^2}\begin{pmatrix}
        (p^{(2)}-1)^2-(p^{(1)})^2 & -2
          p^{(1)} (p^{(2)}-1) \\
        2 p^{(1)} (p^{(2)}-1) & (p^{(2)}-1)^2-(p^{(1)})^2 
       \end{pmatrix}
\end{equation}
and, using that $\Phi$ is an isometry, for any curve $\gamma\colon I\to\D^2$ and $u=\Phi\circ\gamma\colon I\to\H^2$, one has $|\partial_x\gamma(x)|_g = |\partial_x u(x)|_g$ so that, by \eqref{eq:def-Phi}, 
\begin{equation}\label{eq:vel-u-gamma}
    |\partial_x u| = u^{(2)} |\partial_xu|_g = u^{(2)} |\partial_x\gamma|_g = \frac{1-|\gamma|^2}{|\gamma-(0,1)^t|^2} \frac{2|\partial_x\gamma|}{1-|\gamma|^2} = \frac{2}{|\gamma-(0,1)^t|^2}|\partial_x\gamma|.
\end{equation}
Moreover, one immediately finds
\begin{equation}\label{eq:hyp-length}
    \Ll_{\D^2}(\gamma) = \int_I |\partial_x\gamma|_g\dd x = \int_I |\partial_xu|_g\dd x = \Ll_{\H^2}(u)
\end{equation}
and we refer to the common value as the \emph{hyperbolic length} of $\gamma$ or $u$.

The elastic flow of curves in the hyperbolic plane has already been studied in several previous works where the half-plane model $\H^2$ was always used to describe the hyperbolic plane. The half-plane approach turns out to be quite well-suited when studying the flow in the energy regime where one has convergence --- i.e., where no singularities occur. However, as it turns out, when studying singularities of the elastic flow, the Poincar\'e model is better suited for describing the hyperbolic plane. Indeed, in the Poincar\'e model, one does not distinguish between different ``types of $\infty$''. That is, hyperbolic distances blow-up in $\D^2$ if one approaches the boundary of $\D^2$, i.e. $\S^1\subseteq\R^2$. However, in the half-plane model, there are two natural ``types of $\infty$''. Namely, approaching the axis $\{(u^{(1)},0)^t:u^{(1)}\in\R\}$ and approaching infinity in the Euclidean sense. However, in hyperbolic geometry, those two ``types'' are equivalent. Thus, it simplifies the subsequent discussions to work in the Poincar\'e model $\D^2$.

Let now $M$ be either $\H^2$ or $\D^2$ and $I$ either a compact real interval or $\S^1$. Denote by $\nabla$ covariant differentiation and by $\Gamma^{k}_{ij}$ the Christoffel symbols in $M$. Consider an immersion $\gamma\colon I\to M$. Let $V$ be a smooth vector field in $M$ along $\gamma$. Writing $e_1=(1,0)^t$ and $e_2=(0,1)^t$, we define
\begin{equation}\label{eq:cov-der}
    \nabla_xV = \nabla_{\partial_x\gamma} V = \partial_x V + \sum_{i,j,k=1}^2 (\Gamma^k_{ij}\circ\gamma) \partial_x\gamma^{(i)} V^{(j)} e_k.
\end{equation}
For convenience, we also write $\nabla_s$ instead of $\nabla_{\partial_s\gamma}$. Furthermore, for a vector field $V$ in $M$ along $\gamma$, we write $V^{\bot}=V-\langle V,\partial_s\gamma\rangle_g\partial_s\gamma$. In order to make sense of a signed curvature function, define the normal field 
\begin{equation}
    \vec{n} = \begin{pmatrix}
        0& -1\\
        1&0\\
    \end{pmatrix} \cdot \partial_s\gamma.
\end{equation}
Then $\scurv = \langle\curv,\vec{n}\rangle_g$ denotes the signed curvature of $\gamma$ where $\curv=\nabla_s\partial_s\gamma$. 

\begin{remark}
    As in \cite[Equation (2.9)]{eichmanngrunau2019}, one explicitly computes in the case where $M=\H^2$ and $u=\gamma$ that
    \begin{equation}\label{eq:curv-formula}
        \scurv = \frac{\partial_x^2u^{(2)}\partial_xu^{(1)}u^{(2)}-\partial_x^2u^{(1)}\partial_xu^{(2)}u^{(2)}+\partial_xu^{(1)}|\partial_xu|^2}{|\partial_xu|^3}.
    \end{equation}
\end{remark}
With the computations in \cite[Remark 2.5]{dallacquaspener2017}, one finds
\begin{equation}\label{eq:first-var}
    \frac{\dd}{\dd t} \E(\gamma+tV) \Big|_{t=0} = \int_I \langle 2({\nabla_s^{\bot}})^2\curv+|\curv|_g^2\curv - 2\curv,V\rangle_g\dd s =\vcentcolon \int_I \langle \nabla\E(\gamma),V\rangle_g\dd s 
\end{equation}
where $V\colon I\to\R^2$ is smooth and, if $I\neq \S^1$, $V=0$ and $\nabla_s^{\bot}V=0$ on $\partial I$.

The following properties of immersions in the hyperbolic plane are elementary and their proof is given in the appendix for the reader's convenience.

\begin{lemma}\label{lem:basic-facts}
    One has the following.
    \begin{enumerate}[(a)]
        \item For any immersion $\gamma\colon[a,b)\to\D^2$ with $|\gamma(x)|\to 1$ for $x\nearrow b$, we have $\Ll_{\D^2}(\gamma)=\infty$.
        \item If $\gamma\colon I\to \D^2$ is an immersion where $I$ is an interval or $\S^1$, then $\Ll_{\D^2}(\gamma)\geq2\Ll_{\R^2}(\gamma)$.
    \end{enumerate} 
\end{lemma}

\subsection{Energy estimates}\label{subsec:en-est}

Firstly we recall the following lemma.

\begin{lemma}[{\cite[Lemma 2.4]{schlierf2023}}]\label{lem:eich-grun}
    Let $I$ be a compact interval and let $u_n\colon I\to\H^2$ be immersions such that, for $\alpha>0$, $u_n^{(2)}\geq\alpha$ on $\partial I$ for each $n\in\N$. Furthermore, suppose that
    \begin{equation}\label{eq:eich-grun}
        \sup_{n\in\N} \E(u_n) < 8.
    \end{equation}
    Then there exists $c>0$ with $u_n^{(2)}\geq c$ for any $n\in\N$.
\end{lemma}

In view of the previous section, we reformulate a contraposition to \cref{lem:eich-grun} in the setting of the Poincar\'e model $\D^2$ for the hyperbolic plane.

\begin{lemma}\label{rem:eich-grun-poin}
    Let $I$ be a compact interval. Consider immersions $\gamma_n\colon I\to\D^2$ such that, for some $\alpha\in (0,1)$, $|\gamma_n(y)|\leq 1-\alpha$ for all $y\in\partial I$ and $n\in\N$. Then
    \begin{equation}
        \limsup_{n\to\infty}\ (\max_{I}|\gamma_n|) = 1 \implies \limsup_{n\to\infty}\E(\gamma_n)\geq 8.
    \end{equation}
\end{lemma}
\begin{proof}
    Without relabeling, consider any subsequence with $\max_I|\gamma_n|\to 1$ for $n\to\infty$. Fix any sequence of points $(p_n)_{n\in\N}$ with $p_n\in \gamma_n(I)$ and $|p_n|\to 1$. After passing to a further subsequence, we can arrange without loss of generality $p_n\to p_{\infty}\in\partial\D^2$. Again without loss of generality, by applying a rotation to each $\gamma_n$ (which is an isometry of $\D^2$ and thus does not change the elastic energy), we may assume that $p_{\infty}=(0,-1)^t$.

    Define $u_n=\Phi\circ\gamma_n$ with $\Phi$ as in \cref{eq:def-Phi}. Then, for $(x_n)_{n\in\N}\subseteq I$ with $\gamma_n(x_n)=p_n$,
    \begin{equation}
        u_n^{(2)}(x_n) = \Phi(p_n)^{(2)} = \frac{1-|p_n|^2}{|p_n-(0,1)^t|^2} \to \frac04=0
    \end{equation}
    for $n\to\infty$. Moreover, for $y\in\partial I$ and $n\in\N$, 
    \begin{equation}
        u_n^{(2)}(y) = \frac{1-|\gamma_n(y)|^2}{|\gamma_n(y)-(0,1)^t|^2} \geq \frac{\alpha(2-\alpha)}{4}
    \end{equation}
    by assumption. By \cref{lem:eich-grun} applied to the (sub-)sequence $(u_n)_{n\in\N}$,
    \begin{equation}
        \sup_{n\in\N}\E(\gamma_n) = \sup_{n\in\N}\E(u_n) \geq 8,
    \end{equation}
    and the claim follows.
\end{proof}

Next, using the special view on curves in $\H^2$ given by considering their surface of revolution, we deduce the following fundamental estimate from some explicit computations and the Gauss-Bonnet theorem.

\begin{lemma}
    Consider a compact real interval $I$ and an immersion $u\colon I\to\H^2$. Then
    \begin{equation}\label{eq:est-dxu1}
        \int_I \frac{(\partial_xu^{(1)})^2}{|\partial_xu|u^{(2)}}\dd x \leq \E(u) + 4.
    \end{equation}
\end{lemma}
\begin{proof}
    Without loss of generality, suppose that $|\partial_xu|\equiv 1$. Write $\Sigma=I\times \S^1$ and $f_u\colon \Sigma\to\R^3$ for the associated surface of revolution. Note that its Euler characteristic is given by $\chi(\Sigma)=0$. Therefore, since $\partial \Sigma$ consists of two round circles, the Gauss-Bonnet theorem (cf. \cite[Theorem 2 on p.104]{docarmo1994}) yields
    \begin{equation}\label{eq:gau-bon}
        \int_{\Sigma} K\dd \mu_{f_u} \leq 4\pi.
    \end{equation}
    Moreover, by direct computation (cf. \cite[p.10]{dallacquaspener2018}), one has that the Gaussian curvature is given by
    \begin{equation}\label{eq:gauss-curv}
        K = (\partial_x^2u^{(1)}\partial_xu^{(2)}-\partial_x^2u^{(2)}\partial_xu^{(1)})\frac{\partial_xu^{(1)}}{u^{(2)}}.
    \end{equation}
    Therefore, using \cref{eq:curv-formula} and $|\partial_xu|\equiv 1$,
    \begin{align}
        \E(u) &= \int_I (\partial_x^2u^{(2)}\partial_xu^{(1)}u^{(2)}-\partial_xu^{(2)}\partial_x^2u^{(1)}u^{(2)}+\partial_xu^{(1)})^2\frac{1}{u^{(2)}}\dd x \\
        &\geq \int_I \frac{(\partial_xu^{(1)})^2}{u^{(2)}}\dd x + 2 \int_I (\partial_x^2u^{(2)}\partial_xu^{(1)}-\partial_x^2u^{(1)}\partial_xu^{(2)})\partial_xu^{(1)}\dd x\\
        &= \int_I \frac{(\partial_xu^{(1)})^2}{u^{(2)}}\dd x - 2 \int_I K\cdot u^{(2)}\dd x = \int_I \frac{(\partial_xu^{(1)})^2}{u^{(2)}}\dd x - \frac{1}{\pi} \int_{\Sigma} K\dd \mu_{f_u},
    \end{align}
    using $\dd\mu_{f_u}=u^{(2)}\dd x\otimes \dd\varphi$. Now \cref{eq:est-dxu1} follows by \cref{eq:gau-bon}.
\end{proof}

The above estimate is crucial in the following quantization theorem. More precisely, \cref{cor:ch-ve} quantifies how a sequence of immersions can approach ``hyperbolic infinity'' in terms of a bound on the elastic energy. Clearly, this understanding is crucial if one aims to better understand the behavior of singularities of curve evolutions in $\H^2$. Its proof uses a Choksi-Veneroni-type trick, cf. \cite[Step II in proof of Proposition 2]{choksiveneroni2013}. 

\begin{theorem}\label{cor:ch-ve}
    Consider immersions $u_n\colon I\to\H^2$ on a compact interval $I\subseteq\R$ and suppose that $0<\ell\leq |\partial_x u_n|\leq L<\infty$. If $(|\partial_xu_n|)_{n\in\N}$ converges uniformly on $I$,
    \begin{equation}
        u_n\rightharpoonup^*u_{\infty}\text{ in $W^{1,\infty}(I)$}\quad\text{and}\quad u_n\to u_{\infty}\text{ uniformly on $I$}
    \end{equation}
    and
    \begin{equation}\label{eq:ch-ve-1}
        \sup_{n\in\N}\E(u_n) \leq E\cdot 8 + \eta\quad\text{for some $E\in\N_0$ and $\eta<8$},
    \end{equation}
    then $Z=\{u_{\infty}^{(2)}=0\}\cap I^{\circ}$ consists of at most $E$ points.
\end{theorem}
Already in \cite[Lemma~6]{choksiveneroni2013}, an upper bound for the number of elements of $\{u^{(2)}=0\}$ is obtained for a so-called \emph{generalized generator} $u\colon[-1,1]\to\R\times[0,\infty)$, see \cite[Definition~3]{choksiveneroni2013}. In comparison to \cite[Lemma~6]{choksiveneroni2013}, with \eqref{eq:ch-ve-1}, we obtain a new upper bound in the above theorem with a different proof and can allow for curves which are not necessarily parametrized by constant Euclidean speed.
\begin{proof}[Proof of \cref{cor:ch-ve}.]
    As a first step, we argue that $Z$ is made up of at most finitely many connected components $(I_j)_{j=1,\dots,{m}}$ for some $1\leq m\leq E+2$. Indeed, otherwise, there are points $a_1<\dots<a_{E+3}$ all belonging to different connected components of $Z$. In particular, there exist $x_1,\dots,x_{E+2}\in I^{\circ}\setminus Z$ with $a_j<x_j<a_{j+1}$ for all $j=1,\dots,E+2$. Specifically, using the uniform convergence $u_n\to u_{\infty}$, one has $u_n^{(2)}(x_j)\to u_{\infty}^{(2)}(x_j)>0$, $u_n^{(2)}(x_{j+1})\to u_{\infty}^{(2)}(x_{j+1})>0$  and $u_n^{(2)}(a_{j+1})\to u_{\infty}^{(2)}(a_{j+1})=0$ for all $j=1,\dots,E+1$, using the definition of $Z$. So \cref{lem:eich-grun} yields that $\sup_{n\in\N}\E(u_n|_{[x_j,x_{j+1}]})\geq 8$ for $j=1,\dots,E+1$, i.e., $\sup_{n\in\N}\E(u_n)\geq 8(E+1)$ which contradicts \cref{eq:ch-ve-1}. Therefore, $Z=\bigcup_{j=1}^{{m}}I_j$ for the connected components $(I_j)_{j=1,\dots,{m}}$ of $Z$ where $1\leq m\leq E+2$. We later argue that actually $m\leq E$. 
    
    Note that each $I_j$ is an interval as a connected subset of $\R$. So $|Z|=0$ if and only if all connected components contain exactly one point. In order to prove $|Z|=0$, we proceed as follows. For any interval $(a,b)\subseteq I$, using \cref{eq:br-gr},
    \begin{align}
        \Big| 4 \frac{\partial_xu_n^{(2)}}{|\partial_xu_n|}\Big|_a^b \Big| = \Big| \frac{2}{\pi}\W(f_{u_n|_{[a,b]}}) - \E(u_n|_{[a,b]}) \Big| \leq \frac{2}{\pi} \W(f_{u_n|_{[a,b]}}) + \E(u_n|_{[a,b]}).
    \end{align}
    Since $\W(f_{u_n})$ and $\E(u_n)$ are uniformly bounded in $n$ by \cref{eq:ch-ve-1} and \cref{eq:br-gr}, one thus obtains that the variation $V_{I}(\partial_xu_n^{(2)}/|\partial_xu_n|)$ is uniformly bounded in $n$. Therefore, after passing to a subsequence without relabeling, we have $\partial_xu_n^{(2)}/|\partial_xu_n|\to \sigma$ strongly in $L^1(I)$ for some $\sigma\in L^1$. By the hypotheses, there exists $\vartheta\in C^0(I)$ with $|\partial_xu_n|\to\vartheta$ uniformly. The uniform bounds $\ell\leq \vartheta\leq L$ yield
    \begin{equation}
        \partial_xu_n^{(2)} \to \sigma\vartheta\quad\text{in $L^1$}.
    \end{equation}
    Since however $|\partial_xu_n^{(2)}|\leq L$, this immediately improves to $\partial_xu_n^{(2)}\to \sigma\vartheta$ strongly in $L^p(I)$ for any $1\leq p<\infty$. By the hypotheses, we have $\partial_xu_n^{(2)}\rightharpoonup^* \partial_xu_{\infty}^{(2)}$ in $L^{\infty}(I)$. Together, this yields $\sigma\vartheta=\partial_xu_{\infty}^{(2)}$. That is,
    \begin{equation}\label{eq:dir-conv-1}
        \partial_xu_n^{(2)} \to \partial_xu_{\infty}^{(2)}\quad\text{strongly in $L^p(I)$ for any $1\leq p<\infty$}.
    \end{equation}
    Let $\varepsilon>0$ be arbitrary. By the uniform convergence $u_n\to u_{\infty}$, there exists $N(\varepsilon)\in\N$ with $u_n^{(2)}\leq \varepsilon$ on $Z$ for all $n\geq N(\varepsilon)$. In particular, using \cref{eq:est-dxu1,eq:ch-ve-1}, we obtain for $n\geq N(\varepsilon)$
    \begin{equation}
        8E+\eta+4 \geq \E(u_n) + 4 \geq \int_Z \frac{(\partial_xu_n^{(1)})^2}{L\cdot u_n^{(2)}}\dd x \geq \frac{1}{L\cdot \varepsilon} \|\partial_xu_n^{(1)}\|_{L^2(Z)}^2.
    \end{equation}
    Therefore, one has $\partial_x u_n^{(1)}\to 0$ strongly in $L^2(Z)$. Thus, using also \cref{eq:dir-conv-1}, $(\partial_x u_n^{(2)})^2 = |\partial_xu_n|^2-(\partial_x u_n^{(1)})^2\geq \ell^2-(\partial_x u_n^{(1)})^2$ yields $\partial_xu_{\infty}^{(2)}\neq 0$ a.e. on $Z$. However, since $u_{\infty}^{(2)}$ is Lipschitz and minimal everywhere on $Z$, one has $\partial_xu_{\infty}^{(2)}=0$ a.e. on $Z$ by Rademacher's theorem. Alternatively, $\partial_xu_{\infty}^{(2)}=0$ a.e. on $Z$ also follows by a Lemma of De La Vall\'ee Poussin, cf. \cite[Lemma 2.25]{butazzogiaquintahildebrandt1998}. This, \eqref{eq:dir-conv-1} and the obtained convergence $\partial_xu_n^{(1)}\to 0$ in $L^2(Z)$ yields that $|\partial_xu_n|\to 0$ in $L^2(Z)$. Since for all $n\in\N$ one has $\ell\leq |\partial_xu_n|$, we infer that $|Z|=0$. As noted above, this yields $Z=\{a_1,\dots,a_{{m}}\}$. 

    Note that we necessarily have ${m}\leq E$ by \cref{eq:ch-ve-1}, since \cref{lem:eich-grun} yields that, for some choice of $x_1,\dots,x_{{m}+1}\in I^{\circ}$ with $x_j<a_j<x_{j+1}$ for $1\leq j\leq {m}$, $\sup_{n\in\N}\E(u_n|_{[x_j,x_{j+1}]})\geq 8$, similarly as at the beginning of the proof. 
\end{proof}

\subsection{Elastic energy of curves approaching hyperbolic infinity with non-vanishing Euclidean length}\label{subsec:len-est}

In the following, we rigorously show that, if a sequence of curves in $\D^2$ approaches ``hyperbolic infinity'', i.e. $|\gamma_n|\to 1$ uniformly, and if $\E(\gamma_n)\leq M$ for all $n\in\N$, then their Euclidean lengths must converge to $0$. To this end, we firstly deduce the following lemma in the half-plane model $\H^2$ as a direct consequence of \cref{eq:est-dxu1}.

\begin{lemma}\label{cor:bu-elen}
    Consider immersions $u_n\colon I\to\H^2$ and measurable $E_n\subseteq I$ with $\sup_{n\in\N}\Ll_{\R^2}(u_n)<\infty$, $\int_{E_n}|\partial_xu_n^{(1)}|\dd x \geq \ell>0$ and $\|u_n^{(2)}\|_{L^{\infty}(E_n)} \to 0$. Then $\E(u_n) \to \infty$.
\end{lemma}
\begin{proof}
    Fix $\varepsilon>0$. Then for $n$ sufficiently large, $u_n^{(2)}\leq \varepsilon$ a.e. on $E_n$. Without loss of generality, $u_n$ is parametrized by constant Euclidean speed on $I=[0,1]$. That is, $|\partial_xu_n|\equiv \Ll_{\R^2}(u_n)$. Using \cref{eq:est-dxu1} and Cauchy-Schwarz,
    \begin{align}
        \E(u_n)+4&\geq \int_0^1 \frac{(\partial_xu_n^{(1)})^2}{|\partial_xu_n|u_n^{(2)}}\dd x = \frac{1}{\Ll_{\R^2}(u_n)}\Bigl\| \frac{\partial_xu_n^{(1)}}{{u_n^{(2)}}^{\frac{1}{2}}} \Bigr\|_{L^2(I)}^2  \geq \frac{1}{\Ll_{\R^2}(u_n)} \Bigl({\int_{E_n} \frac{|\partial_xu_n^{(1)}|}{{u_n^{(2)}}^{\frac{1}{2}}} \dd x}\Bigr)^2\\
        &\geq \frac{1}{\Ll_{\R^2}(u_n)\cdot\varepsilon} \Bigl({\int_{E_n} |\partial_xu_n^{(1)}|\dd x}\Bigr)^2 \geq \frac{\ell^2}{\Ll_{\R^2}(u_n)\cdot\varepsilon}. \qedhere 
    \end{align}
\end{proof}

The following proposition generalizes the result of \cref{cor:bu-elen} in $\H^2$ from a lower bound on Euclidean movements only in the horizontal direction (i.e. of $u_n^{(1)}$) to a lower bound on the full Euclidean length of $u_n$. The main idea in its proof is the observation that, if $u_n^{(2)}\to 0$ uniformly and if $\int_I |\partial_xu_n^{(1)}|\dd x\to 0$ while $\Ll_{\R^2}(u_n)\geq \ell$, then each $u_n$ has to have an in $n$ unbounded amount of individual segments with $\partial_xu_n^{(2)}>0$ and $\partial_xu_n^{(2)}<0$. For each of these segments, we deduce a lower bound on the elastic energy using \eqref{eq:br-gr}. Since the proof is rather technical, it is postponed to \cref{app:el}.

\begin{proposition}\label{cor:bu-len}
    Consider immersions $u_n\colon I\to\H^2$ such that $\|u_n^{(2)}\|_{L^{\infty}(I)}\to 0$ and suppose that $\Ll_{\R^2}(u_n)\geq \ell>0$. Then $\limsup_{n\to\infty}\E(u_n)=\infty$.
\end{proposition}

\begin{corollary}\label{cor:bu-en-poin}
    Let $\gamma_n\colon I\to\D^2$ be immersions with the property that $\Ll_{\R^2}(\gamma_n|_{E_n})\geq \ell>0$ and $\sup_{E_n}|1-|\gamma_n||\to 0$ for intervals $E_n\subseteq I$. Then $\E(\gamma_n)\to\infty$ for $n\to\infty$.
\end{corollary}
\begin{proof}
    Without loss of generality, write $I=[0,1]$. By restricting to intervals $E_n'\subseteq E_n$, one can arrange for some fixed $\ell'\in(0,\min\{\ell,\frac14\})$ that
    \begin{equation}\label{eq:lenbound}
        0<\ell' \leq \Ll_{\R^2}(\gamma_n|_{E_n'}) \leq 2\ell' < \frac{1}{2}
    \end{equation}
    for all $n\in\N$. Denote by $\widetilde{\gamma_n}\colon I\to\D^2$ the reparametrization of $\gamma_n|_{E_n'}$ by constant Euclidean speed on $I$. By the parametrization and \cref{eq:lenbound} and since $\D^2$ is bounded in $\R^2$, after passing to a subsequence, $\widetilde{\gamma_n}\rightharpoonup^*\widetilde{\gamma_{\infty}}$ in $W^{1,\infty}(I)$. Lower semi-continuity yields that
    \begin{equation}
        \| |\partial_x\widetilde{\gamma_{\infty}}| \|_{L^{\infty}(I)} \leq \liminf_{n\to\infty} \| |\partial_x\widetilde{\gamma_n}| \|_{L^{\infty}(I)} = \liminf_{n\to\infty} \Ll_{\R^2}(\widetilde{\gamma_n}) < \frac{1}{2}
    \end{equation}
    using \cref{eq:lenbound}. Therefore, $\Ll_{\R^2}(\widetilde{\gamma_{\infty}}) \leq \frac{1}{2}$. Furthermore, by the assumptions in the statement, $|\widetilde{\gamma_{\infty}}|\equiv 1$, that is $\widetilde{\gamma_{\infty}}(I)\subseteq\partial\D^2=\S^1$ has Euclidean length at most $\frac12$. 

    Therefore, there exists a rotation $O\in\R^{2\times 2}$ such that $\mathrm{dist}_{\mathbb{R}^2}(O\cdot \widetilde{\gamma_{\infty}}(I),(0,1)^t)>\frac{1}{2}$. Since $O$ is an isometry of $\D^2$ and thus keeps the elastic energy fixed, we can suppose without loss of generality that $\mathrm{dist}_{\mathbb{R}^2}(\widetilde{\gamma_{\infty}}(I),(0,1)^t)>\frac12$. 

    By the uniform convergence $\widetilde{\gamma_n}\to\widetilde{\gamma_{\infty}}$, there exists $N_0\in\N$ with $\mathrm{dist}_{\mathbb{R}^2}(\widetilde{\gamma_n}(I),(0,1)^t)>\frac14$ for $n\geq N_0$. Then, writing $u_n=\Phi\circ\widetilde{\gamma_n}\colon I\to\H^2$ with $\Phi$ as in \cref{eq:def-Phi}, using \cref{eq:vel-u-gamma},
    \begin{equation}
        \Ll_{\R^2}(u_n)=\int_I |\partial_xu_n|\dd x = 2 \int_I \frac{|\partial_x\widetilde{\gamma_n}|}{|\widetilde{\gamma_n}-(0,1)^t|^2}\dd x,
    \end{equation}
    so $\Ll_{\R^2}(u_n)\geq \frac12 \Ll_{\R^2}(\widetilde{\gamma_n}) \geq \frac12\ell'$ for all $n\geq N_0$, using also \cref{eq:lenbound}. Moreover, for $n\geq N_0$,
    \begin{equation}
        u_n^{(2)} = \frac{1-|\widetilde{\gamma_n}|^2}{|\widetilde{\gamma_n}-(0,1)^t|^2} \leq 16 (1-|\widetilde{\gamma_n}|^2) 
    \end{equation}
    which converges uniformly to $0$ for $n\to\infty$ by assumption. Altogether, \cref{cor:bu-len} yields that $\limsup_{n\to\infty}\E(\gamma_n)\geq \limsup_{n\to\infty}\E(\widetilde{\gamma_n}) = \limsup_{n\to\infty}\E(u_n) = \infty$. The claim then follows by a standard sub-sequence argument.
\end{proof}

As explained before, in this work, we focus on the Poincar\'e model $\D^2$ for the hyperbolic plane. Thus, we deduce the following analogous formulation of \cref{cor:ch-ve} in the Poincar\'e model, using \cref{cor:bu-en-poin}. 

\begin{corollary}\label{cor:ch-ve-poin}
    Consider immersions $\gamma_n\colon I\to\D^2$ on a compact interval $I\subseteq\R$ and suppose that $0<\ell\leq |\partial_x \gamma_n|\leq L<\infty$. If $(|\partial_x\gamma_n|)_{n\in\N}$ converges uniformly on $I$,
    \begin{equation}
        \gamma_n\rightharpoonup^*\gamma_{\infty}\text{ in $W^{1,\infty}(I)$}\quad\text{and}\quad \gamma_n\to \gamma_{\infty}\text{ uniformly on $I$}
    \end{equation}
    and
    \begin{equation}\label{eq:ch-ve-poi}
        \sup_{n\in\N}\E(\gamma_n) \leq E\cdot 8 + \eta\quad\text{for some $E\in\N_0$ and $\eta<8$},
    \end{equation}
    then $Z=\{|\gamma_{\infty}|=1\}\cap I^{\circ}$ consists of at most $E$ points.
\end{corollary}
\begin{proof}
    Clearly, this is proved by applying \cref{cor:ch-ve}. However, we first need to better understand the sequence $(\gamma_n)_{n\in\N}$ in order to be able to verify the assumptions of \cref{cor:ch-ve} for the sequence $u_n=\Phi\circ\gamma_n$ with $\Phi$ as in \cref{eq:def-Phi}.

    \textbf{Step 1.} Using \cref{rem:eich-grun-poin}, one argues just as in the proof of \cref{cor:ch-ve} that $Z=\bigcup_{j=1}^{{m}}I_j$ for finitely many connected components $I_1,\dots,I_{{m}}$ of $Z$.

    \textbf{Step 2.} Argue that ${\gamma_{\infty}}(I_j)$ consists of exactly one point in $\partial \D\subseteq\R^2$ for each $j=1,\dots,{m}$. Suppose that this is not the case for some $1\leq j\leq {m}$. In particular, there exist $p_1,p_2\in {\gamma_{\infty}}(I_j)$ with $\delta\vcentcolon=|p_1-p_2|>0$. Due to the uniform convergence ${\gamma_n}\to{\gamma_{\infty}}$, there exists $N_0\in\N$ such that, for each $n\geq N_0$, there are points $\xi_n^{(1)},\xi_n^{(2)}\in I_j$ with $|{\gamma_n}(\xi_n^{(i)})-p_i|\leq \frac14\delta$ for $i=1,2$. Therefore,
    \begin{equation}
        \Ll_{\R^2}({\gamma_n}|_{I_j}) \geq |{\gamma_n}(\xi_n^{(1)})-{\gamma_n}(\xi_n^{(2)})| \geq \frac12\delta
    \end{equation}
    for $n\geq N_0$. Using that $|{\gamma_{\infty}}|\equiv 1$ on $I_j$, \cref{cor:bu-en-poin} yields $\limsup_{n\to\infty}\E({\gamma_n}|_{I_j})=\infty$, a contradiction! This verifies Step 2.
    
    \textbf{Step 3.} Applying \cref{cor:ch-ve}. Using Step 2 and the uniform convergence ${\gamma_n}\to{\gamma_{\infty}}$, after applying a rotation to $\D^2$, without loss of generality one has for some $\varepsilon>0$ that
    \begin{equation}\label{eq:cvp-1}
        |{\gamma_n}(x)-(0,1)^t|\geq \varepsilon\quad\text{for all $n\in\N$ sufficiently large and $x\in I$.}
    \end{equation}
    Write  $u_n=\Phi\circ{\gamma_n}$ with $\Phi$ as in \cref{eq:def-Phi}. By \cref{eq:vel-u-gamma},
    \begin{equation}
        |\partial_xu_n| = 2\frac{|\partial_x{\gamma_n}|}{|{\gamma_n}-(0,1)^t|^2} 
    \end{equation}
    which, using \cref{eq:cvp-1}, converges uniformly on $I$. Moreover, $\frac12\ell\leq |\partial_xu_n| \leq \frac{2L}{\varepsilon^2}$ on $I$ for $n$ sufficiently large. In addition, again using \cref{eq:cvp-1} and \cref{eq:def-Phi}, $(u_n)_{n\in\N}$ converges uniformly to $u_{\infty} \vcentcolon=\Phi\circ \gamma_{\infty}$. The only assumption of \cref{cor:ch-ve} which remains to be verified is that $\partial_xu_n\rightharpoonup^*\partial_xu_{\infty}$ in $L^{\infty}(I)$. To this end, using \cref{eq:Jacobi-Phi}, one computes that 
    \begin{equation}
        \partial_xu_n^{(1)} = 2\frac{(\gamma^{(2)}-1)^2-(\gamma^{(1)})^2}{|\gamma-(0,1)^t|^4}\partial_x\gamma^{(1)} - 4\frac{\gamma^{(1)}(\gamma^{(2)}-1)}{|\gamma-(0,1)^t|^4}\partial_x\gamma^{(2)},
    \end{equation}
    \begin{equation}
        \partial_xu_n^{(2)} =4 \frac{\gamma^{(1)}(\gamma^{(2)}-1)}{|\gamma-(0,1)^t|^4}\partial_x\gamma^{(1)} + 2\frac{(\gamma^{(2)}-1)^2-(\gamma^{(1)})^2}{|\gamma-(0,1)^t|^4}\partial_x\gamma^{(2)}.
    \end{equation}
    Using \cref{eq:cvp-1}, $\gamma_n\to\gamma_{\infty}$ uniformly and $\partial_x\gamma_n^{(i)}\rightharpoonup^*\partial_x\gamma_{\infty}^{(i)}$ in $L^{\infty}(I)$ for $i=1,2$, it follows that $\partial_xu_n\rightharpoonup^*\partial_xu_{\infty}$ in $L^{\infty}(I)$.
    
    Altogether, \cref{cor:ch-ve} yields the claim, using $\{u_{\infty}^{(2)}=0\}=\{|\gamma_{\infty}|=1\}$ by \cref{eq:cvp-1}. 
\end{proof}

\subsection{On critical points of the elastic energy in the hyperbolic plane}

In this section, $M$ denotes either $\H^2$ or $\D^2$.

\begin{definition}\label{def:elastica}
    A smooth immersion $\gamma\colon I\to M$ is called \emph{$\lambda$-constrained (hyperbolic) elastica} if its scalar curvature $\scurv$ satisfies
    \begin{equation}\label{eq:elastica-eq}
        2\partial_s^2\scurv+\scurv^3-(\lambda+2)\scurv = 0
    \end{equation}
    for some $\lambda\in\R$. If $\scurv$ satisfies \cref{eq:elastica-eq} with $\lambda=0$, $\gamma$ is also referred to as \emph{free elastica} or simply \emph{elastica}.
\end{definition}

\begin{remark}
    Note that, for $\gamma\colon I\to M$ smooth and parametrized by hyperbolic arc-length, \cref{eq:elastica-eq} is equivalent to  $\nabla\E(\gamma) = \lambda\curv$. So free elastica are exactly the critical points of $\E$.
\end{remark}

The following two results of \cite{langersinger1984,muellerspener2020} completely characterize hyperbolic elastica:

\begin{lemma}[{\cite[Proposition 2.7]{muellerspener2020}}]
    Let $\gamma\colon I\to M$ be a $\lambda$-constrained elastica which is parametrized by arc-length in $M$. Then there is a constant $C\in\R$ such that
    \begin{equation}\label{eq:char-el-0}
        (\scurv')^2+\frac{1}{4}\scurv^4-\frac{\lambda+2}{2}\scurv^2=C
    \end{equation}
    and $\zeta\defeq\scurv^2$ is a non-negative solution of 
    \begin{equation}\label{eq:char-el-1}
        (\zeta')^2+\zeta^3-(2\lambda+4)\zeta^2-4C\zeta=0.
    \end{equation}
\end{lemma}

In the following proposition, standard notation for the Jacobi-elliptic functions $\dn$ and $\cn$ is used. 

\begin{proposition}[{\cite[Proposition 2.8]{muellerspener2020}}]\label{prop:2.8}
    Each non-negative solution $\zeta$ of \cref{eq:char-el-1} is global and attains a global maximum $\scurv_0^2\defeq\sup_{x\in\R}\zeta(x)$. Therefore, all non-negative solutions of \cref{eq:char-el-1} are translations of solutions with the following initial conditions: $\zeta(0)=\scurv_0^2$ and $\zeta'(0)=0$. Furthermore, for $0<\scurv_0^2<\lambda+2$, there exist no $\lambda$-constrained elastica and the cases with $\scurv_0^2\geq \lambda+2$ are exhaustively classified by the following:
    \begin{enumerate}[(a)]
        \item (Circular elastica) $\scurv_0^2=\lambda+2$, $C<0$ and $\zeta(s)=\lambda+2$.
        \item (Orbit-like elastica) $\scurv_0^2\in(\lambda+2,2\lambda+4)$, $C<0$ and $\zeta(s)=\scurv_0^2\dn^2(rs,p)$ where $r=\frac{1}{2}\sqrt{\frac{2\lambda+4}{2-p^2}}$ and $p\in(0,1)$ is such that $\scurv_0^2=\frac{2\lambda+4}{2-p^2}$.
        \item (Asymptotically geodesic elastica) $\scurv_0^2=2\lambda+4$, $C=0$ and $\zeta(s)=\scurv_0^2\frac{1}{\cosh^2(rs)}$ where $r=\frac{1}{2}\sqrt{2\lambda+4}$.
        \item (Wave-like elastica) $\scurv_0^2>2\lambda+4$, $C>0$ and $\zeta(s)=\scurv_0^2\cn^2(rs,p)$ where $r=\frac{1}{2}\sqrt{\frac{2\lambda+4}{2p^2-1}}$ and $p\in(\frac{1}{\sqrt{2}},1)$ is such that $\scurv_0^2=\frac{(2\lambda+4)p^2}{2p^2-1}$.
    \end{enumerate} 
\end{proposition}

\begin{remark}
    By the fundamental theorem of curve theory in $\H^2$ and since solutions to \cref{eq:char-el-1} are global, one obtains that any arc-length parametrized elastica defined on an interval can be extended to a globally defined elastica, cf. \cite[Lemma 5.9]{schlierf2023}. 
\end{remark}

\begin{remark}\label{rem:on-as-geo-el}
    The following property of segments of free, asymptotically geodesic elastica is used later. If $\gamma\colon [0,1)\to\H^2$ is asymptotically geodesic with $\partial_s\gamma(0)=(0,\pm1)^t$ and $\lim_{x\nearrow 1}\gamma(x)\in \R\times\{0\}$, then $\gamma^{(1)}(0)\neq \lim_{x\nearrow 1}\gamma^{(1)}(x)$. We briefly outline an argument for this observation in \cref{app:el}.
\end{remark}

\begin{remark}\label{rem:el-sm-elen}
    Suppose that $\gamma\colon I\to M$ is a free elastica with unbounded hyperbolic length and finite elastic energy. Using the explicit formulae for $\scurv$ in \cref{prop:2.8} and the fact that the Jacobi-elliptic functions $\dn$, $\cn$ are periodic, one immediately sees that segments of circular, orbit-like and wave-like elastica with unbounded hyperbolic length also have infinite elastic energy. Thus, $\gamma$ either is a geodesic or a segment of an asymptotically geodesic elastica.
\end{remark}

The proof of the following lemma is postponed to \cref{app:el}. The lemma highlights an advantage of considering $\D^2$ instead of $\H^2$ as model for the hyperbolic plane --- the statement fails in $\H^2$. 

\begin{lemma}\label{lem:el-fin-len}
    There is no elastica in $\D^2$ with finite elastic energy and infinite Euclidean length.
\end{lemma}

\section{Energy quantization of singular limits of the elastic flow}\label{sec:proof-main}

Recall that the elastic flow for closed curves and the clamped elastic flow in the hyperbolic plane are already defined in \cref{eq:clo-ef,eq:cla-ef} with \cref{eq:first-var}, respectively. As a feature of both flows, solutions are always global. Indeed, by \cite[Theorem 4.1]{dallacquaspener2017}, maximal solutions to \cref{eq:clo-ef} exist and are always global. Moreover, the existence of a maximal solution for \cref{eq:cla-ef} can be shown as in \cite{ruppspener2020}. Using the energy estimates in \cite[proof of Theorem~3.14]{schlierf2023}, one argues as in \cite[proof of Theorem~4.1]{dallacquaspener2017} that also any maximal solution to \cref{eq:cla-ef} is global in time. 

The singularities we investigate in this article thus form in infinite time in the sense that the hyperbolic length of $\gamma(t)$ is unbounded as $t\to\infty$, or, equivalently, that the solutions $\gamma(t)$ do not converge (after reparametrization) to a smooth elastica for $t\to\infty$.

\begin{remark}\label{rem:hyp-len-lb}
    For any $\gamma$ solving \cref{eq:clo-ef} or \cref{eq:cla-ef}, one has $\inf_{t\geq 0} \Ll_{\D^2}(\gamma(t))>0$.

    For \cref{eq:clo-ef} note that, by Fenchel's theorem, cf. \cite[Theorem 2.3]{dallacquaspener2017}, one has that $2\pi\leq \int_{\S^1} |\curv(t)|_g\dd s \leq \sqrt{\E(\gamma(t))} \sqrt{\Ll_{\D^2}(\gamma(t))}$ so that, using $\partial_t\E(\gamma(t))\leq 0$ by \cref{eq:first-var},
    \begin{equation}\label{eq:fenchel}
        \Ll_{\D^2}(\gamma(t))\geq \frac{4\pi^2}{\E(\gamma(t))}\geq \frac{4\pi^2}{\E(\gamma_0)}.
    \end{equation}
    Regarding \cref{eq:cla-ef}, this is somewhat more technical. In view of \eqref{eq:hyp-length}, we may work in $\H^2$ and briefly outline the argument for an immersion $u\colon[-1,1]\to\H^2$. If $u(-1)\neq u(1)$, then $\Ll_{\H^2}(u)\geq \mathrm{dist}_{\H^2}(u(-1),u(1))$. Moreover, if $u(-1)=u(1)$ and $\partial_su(-1)=\partial_su(1)$, then $u\in W^{2,2}(\S^1)$ and one can apply the above argument involving Fenchel's theorem after approximation with smooth immersions of $\S^1$. 
    
    Finally, if $u(-1)=u(1)=p\in\H^2$ but $\partial_su(-1)\neq\partial_su(1)$, let $\eta=|\partial_su(-1)-\partial_su(1)|_g$ and $v=\frac{1}{\eta \cdot p^{(2)}}(\partial_su(-1)-\partial_su(1))$. Then $|v|=1$ and, since $p^{(2)}=u^{(2)}(\pm1)$ and $\curv=\nabla_s\partial_su$, 
    \begin{align}
        \eta&=p^{(2)}|\langle\partial_su(-1),v\rangle_g - \langle\partial_su(1),v\rangle_g| \leq \int_{-1}^1 |\langle \curv, u^{(2)}v\rangle_g + \langle \partial_su,\nabla_s (u^{(2)}v)\rangle_g|\dd s \\
        &\leq \sqrt{\Ll_{\H^2}(u)}\sqrt{\E(u)}+\Ll_{\H^2}(u) \leq \frac12\eta + \Big(\frac{\E(u)}{2\eta}+1\Big)\Ll_{\H^2}(u),
    \end{align}
    using the fundamental theorem of calculus, compatibility of the connection $\nabla$ with the metric $g$ and $|\nabla_s (u^{(2)}v)|_g\leq 1$ by the explicit version of \eqref{eq:cov-der} in \cite[Equation~(10)]{dallacquaspener2017}. Indeed, for the last estimate, one computes that the only non-vanishing Christoffel symbols in $\H^2$ are $\Gamma^1_{12}\circ u=\Gamma^1_{21}\circ u=-\frac{1}{u^{(2)}}$ and $\Gamma^2_{11}\circ u=-\Gamma^2_{22}\circ u=\frac{1}{u^{(2)}}$ so that 
    \begin{equation}
        \nabla_s(u^{(2)}v)=\partial_s(u^{(2)}v) +v^{(1)} \begin{pmatrix}
            -\partial_su^{(2)}\\\partial_su^{(1)} 
        \end{pmatrix}- v^{(2)} \begin{pmatrix}
            \partial_su^{(1)}\\\partial_su^{(2)} 
        \end{pmatrix}=\partial_su^{(1)}\begin{pmatrix}
            -v^{(2)}\\v^{(1)}
        \end{pmatrix}. 
    \end{equation}
\end{remark}
 
The following remark shows how \cref{rem:hyp-len-lb} can be applied to obtain lower bounds on the Euclidean length in $\D^2$ as long as the curve stays away from the boundary $\partial\D^2$ at least at one point.

\begin{remark}\label{rem:euc-len-lb}
    Consider a sequence $(\gamma_n)_{n\in\N}$ of immersions of an interval or $\S^1$ into $\D^2$ with $\Ll_{\D^2}(\gamma_n)\geq \ell>0$. Furthermore, suppose that there exists a sequence $(x_n)_{n\in\N}$ such that $|\gamma_n(x_n)|^2<1-\varepsilon$ for some $\varepsilon>0$. Then $\inf_{n\in\N}\Ll_{\R^2}(\gamma_n)>0$.

    Indeed, if $\Ll_{\R^2}(\gamma_n)\to 0$ for a subsequence (without relabeling), then there exists $N_0\in\N$ such that, for $n\geq N_0$, $|\gamma_n|^2\leq 1-\frac12\varepsilon$. Consequently,
    \begin{equation}
        \Ll_{\R^2}(\gamma_n) = \int |\partial_x\gamma_n|\dd x = \int \frac{1-|\gamma_n|^2}{2} |\partial_x\gamma_n|_g\dd x \geq \frac{\varepsilon}{4} \ell,
    \end{equation}
    a contradiction!
\end{remark}

\subsection{Preliminaries: Meta results}

The following lemma is rather technical in its general form, but proves to be a useful tool and is applied at several points throughout this article. As it turns out, reparametrizations of curves in $\D^2$ or $\H^2$ by constant \emph{Euclidean} speed are quite practical when studying singularities. In fact, if the hyperbolic geometry blows up, then there is no reason why some Euclidean quantities should not behave better. Thus, we first introduce the following definition. 

\begin{definition}[Reparametrization by constant \emph{Euclidean} speed]
    Let $I,J\subseteq\R$ be compact intervals, $|I|,|J|>0$, and $\gamma\colon I\to\R^2$ an immersion. Then the reparametrization of $\gamma$ by constant Euclidean speed on $J$ is given by $\widetilde{\gamma}\colon J\to \R^2$, $\widetilde{\gamma}(x)=\gamma(\varphi(x))$ where $\varphi\colon J\to I$ is a smooth diffeomorphism and $\varphi^{-1}(y)=\min(J)+\frac{|J|}{\Ll_{\R^2}(\gamma)} \int_{I\cap (-\infty,y]} |\partial_x\gamma|\dd x$. 

    Moreover, if $\gamma\colon\S^1\to\R^2$, then identify $\S^1=\nicefrac{I\ }{\sim}$ for some compact interval $I\subseteq\R$ where the endpoints are equivalent by ``$\sim$'' and apply the above construction with $J=I$ to obtain $\widetilde{\gamma}\colon\S^1\to\R^2$. 
\end{definition}

\begin{lemma}[Compactness]\label{lem:int-lem}
    Denote by $I$ either a compact interval or $\S^1$ and let $\gamma\colon[0,\infty)\times I\to \D^2$ be a solution to either \cref{eq:clo-ef} or \cref{eq:cla-ef}. For sub-intervals $J_n\subseteq I$ and isometries $\phi_n\colon\D^2\to\D^2$ and $t_n\nearrow \infty$, write $\gamma_n=\phi_n\circ \gamma(t_n,\cdot)|_{J_n}$. Choose a compact interval $I'\subseteq\R$ with $|I'|>0$ and denote by $\widetilde{\gamma_n}\colon I'\to\D^2$ the reparametrizations of $\gamma_n$ by constant Euclidean speed.

    If $0<\ell\leq\Ll_{\R^2}(\gamma_n)\leq L<\infty$ for all $n\in\N$, then there exists $\widetilde{\gamma_{\infty}}\in W^{1,\infty}(I',\overline{\D^2})$ such that, after passing to a subsequence, $\widetilde{\gamma_n} \rightharpoonup^* \widetilde{\gamma_{\infty}} $ in $W^{1,\infty}(I',\R^2)$ and strongly in $C^{\alpha}(I')$ for any $\alpha\in(0,1)$. Furthermore, writing $Z=\{|\widetilde{\gamma_{\infty}}|=1\}$, one has $\widetilde{\gamma_{\infty}}|_{I'\setminus Z}\in C^{\infty}$ and $\widetilde{\gamma_n}\to\widetilde{\gamma_{\infty}}$ in $C^{\infty}_{\mathrm{loc}}(I'\setminus Z)$\footnote{By ${(\cdot)}_{\mathrm{loc}}$ we mean convergence on compact subintervals.}.
    
    Lastly, $\nabla\E(\widetilde{\gamma_{\infty}})=0$ pointwise on $I'\setminus Z$.
\end{lemma}
\begin{proof}
    By assumption, we have that $|\widetilde{\gamma_n}|\leq 1$ and, due to the parametrization by constant Euclidean speed on $I'$,
    \begin{equation}\label{eq:param-bd}        
        |\partial_x\widetilde{\gamma_n}| \equiv \frac{\Ll_{\R^2}(\widetilde{\gamma_n})}{|I'|} = \frac{\Ll_{\R^2}({\gamma_n})}{|I'|} \leq \frac{L}{|I'|},
    \end{equation}
    i.e., $(\widetilde{\gamma_n})_{n\in\N}\subseteq W^{1,\infty}(I')$ is uniformly bounded. Thus, there exists $\widetilde{\gamma_{\infty}}\in W^{1,\infty}(I',\R^2)$ such that, after passing to a subsequence without relabeling, $\widetilde{\gamma_n}\rightharpoonup^*\widetilde{\gamma_{\infty}}$ in $W^{1,\infty}(I')$ and $\widetilde{\gamma_n}\to\widetilde{\gamma_{\infty}}$ in $C^{\alpha}(I')$. In particular, $|\widetilde{\gamma_{\infty}}|\leq 1$ on $I'$. 

    In order to show smoothness of $\widetilde{\gamma_{\infty}}$ on $I'\setminus Z$ and convergence in $C^{\infty}_{\mathrm{loc}}(I'\setminus Z)$, we first recall the following energy estimates for derivatives of the curvature $\curv(t)$ of $\gamma(t)$ which can be obtained using interpolation techniques. One has
    \begin{equation}\label{eq:ines}
        \sup_{t\in[0,\infty)} \int_I |({\nabla_s^{\perp}})^m\curv|_g^2(t)\dd s + \||\nabla_{s}^m\curv|_g(t)\|_{\infty} < \infty \quad\text{for all $m\in\N_0$}.
    \end{equation}
    Indeed, if $I=\S^1$, i.e. for the closed elastic flow, this is \cite[(53) and (55)]{dallacquaspener2017}. For the clamped flow, see \cite[(3.41) and (3.43) with $a\equiv -2$]{schlierf2023}. The $L^{\infty}$-estimate in the case $m=0$ follows by interpolation using \cite[Proposition 4.1]{dallacquaspener2017} and only requires a uniform lower bound on the hyperbolic length of $\gamma(t)$ which is satisfied, cf. \cref{rem:hyp-len-lb}.

    Furthermore, by \cite[Lemma 2.7]{dallacquaspener2017}, for a normal vector field $N$ along $\widetilde{\gamma_n}$, one has that
    \begin{equation}\label{eq:lem2.7}
        \nabla_{x}^mN = |\partial_x\widetilde{\gamma_n}|_g^m\nabla_{s}^mN + \sum_{j=1}^{m-1} P_{m,j}(|\partial_x\widetilde{\gamma_n}|_g,\dots,\partial_x^{m-j}|\partial_x\widetilde{\gamma_n}|_g)\nabla_{s}^jN
    \end{equation}
    where $P_{m,j}$ is a polynomial of degree at most $m-1$. Moreover, 
    \begin{equation}\label{eq:der-g}
        \partial_x^m |\partial_x\widetilde{\gamma_n}|_g = \partial_x^m (\frac{2}{1-|\widetilde{\gamma_n}|^{2}}|\partial_x\widetilde{\gamma_n}|) = \frac{\Ll_{\R^2}(\widetilde{\gamma_n})}{|I'|} \partial_x^m \frac{2}{1-|\widetilde{\gamma_n}|^2}        
    \end{equation}
    is a polynomial in $\frac{1}{1-|\widetilde{\gamma_n}|^2}$ as well as $\widetilde{\gamma_n},\partial_x\widetilde{\gamma_n},\dots,\partial_x^m\widetilde{\gamma_n}$.
    
    Fix any $0<\varepsilon<1$ and define $I'_{\varepsilon}=\{|\widetilde{\gamma_{\infty}}|^2<1-\varepsilon\}$. Note that showing convergence in $C^{\infty}_{\mathrm{loc}}(I'\setminus Z)$ is equivalent to showing convergence in $C^{\infty}_{\mathrm{loc}}(I_{\varepsilon}')$ for all $0<\varepsilon<1$.

    The first step is showing that $|\partial_x^2\widetilde{\gamma_n}|$ is uniformly bounded. To this end, note that
    \begin{align}        
        \curv_n &= \nabla_s\partial_s\widetilde{\gamma_n} =  \frac{|I'|}{\Ll_{\R^2}(\gamma_n)}\frac{1-|\widetilde{\gamma_n}|^2}{2} \nabla_{x}\Bigl(\frac{|I'|}{\Ll_{\R^2}(\gamma_n)}\frac{(1-|\widetilde{\gamma_n}|^2)\partial_x\widetilde{\gamma_n}}{2}\Bigr)\\
        &=  \Bigl( \frac{|I'|}{\Ll_{\R^2}(\gamma_n)}\frac{1-|\widetilde{\gamma_n}|^2}{2} \Bigr)^2 \nabla_x\partial_x\widetilde{\gamma_n} -  \Bigl( \frac{|I'|}{\Ll_{\R^2}(\gamma_n)}\Bigr)^2 \frac{1-|\widetilde{\gamma_n}|^2}{2} \langle\widetilde{\gamma_n},\partial_x\widetilde{\gamma_n}\rangle \partial_x\widetilde{\gamma_n}
    \end{align}
    and, by \cref{eq:ines}, $|\curv_n|\leq |\curv_n|_g \leq \||\curv(t_n,\cdot)|_g\|_{\infty} \leq C$, using the invariance of the hyperbolic norm of the curvature under isometries of $\D^2$ and reparametrizations. Since $\widetilde{\gamma_n}\to\widetilde{\gamma_{\infty}}$ uniformly, one has that $1-|\widetilde{\gamma_n}|^2>\frac{\varepsilon}{2}$ on $I'_{\varepsilon}$ for $n$ sufficiently large. For such values of $n$, using \cref{eq:param-bd}, one finds 
    \begin{align}
        |\nabla_x\partial_x\widetilde{\gamma_n}|\leq \frac{16}{\varepsilon^2}\Bigl(\frac{L}{|I'|}\Bigr)^2\ (C + 1)\quad\text{on $I'_{\varepsilon}$}
    \end{align}
    so that particularly also $\sup_{n\in\N} \|\partial_x^2\widetilde{\gamma_n}\|_{L^{\infty}(I'_{\varepsilon})} < \infty$, using \cref{eq:cov-der}.

    With \cref{eq:lem2.7} applied to $N=\curv_n$ with $m-2$ in place of $m$, using also \cref{eq:der-g}, analogous computations inductively yield $\sup_{n\in\N}\|\partial_x^m\widetilde{\gamma_n}\|_{L^{\infty}(I'_{\varepsilon})} < \infty$ from \cref{eq:param-bd,eq:ines} for all $m\geq 3$. Therefore, the smooth convergence on compact intervals in $I'_{\varepsilon}$ follows. 
    
    We now argue that $|\partial_x\widetilde{\gamma_{\infty}}|>0$ on $I'\setminus Z$, i.e. $\widetilde{\gamma_{\infty}}$ is an immersion on $I'\setminus Z$. Indeed, by the parametrization of $\widetilde{\gamma_{n}}$ and the smooth convergence on compact intervals in $I'\setminus Z$, $|\partial_x\widetilde{\gamma_{\infty}}|\equiv \lim_{n\to\infty}\Ll_{\R^2}(\widetilde{\gamma_n})/|I'|\geq\ell/|I'|>0$ is constant in $I'\setminus Z$.

    Lastly, we verify the property that $\nabla\E(\widetilde{\gamma_{\infty}})=0$ on $I'\setminus Z$. To this end, using \cref{eq:first-var}, define $h\colon[0,\infty)\to\R$ by 
    \begin{equation}\label{eq:lim-is-el-poin}
        h(t) = \int_{I} |\nabla\E(\gamma(t))|_g^2\dd s = - \partial_t\E(\gamma(t))\in L^1(0,\infty).
    \end{equation}
    As in \cite[p.22]{dallacquaspener2017}, one uses the interpolation inequality in \cite[Proposition 4.3]{dallacquaspener2017} and \cref{eq:ines} to show $\partial_th\in L^{\infty}(0,\infty)$. Therefore, $h(t)\to 0$ for $t\to\infty$. So, using again the invariance of $\nabla\E$ and $|\cdot|_g$ with respect to isometries of $\D^2$ and reparametrizations,
    \begin{align}
        \int_{I'_{\varepsilon}} |\nabla\E(\widetilde{\gamma_{\infty}})|_g^2\dd s &= \lim_{n\to\infty} \int_{I'_{\varepsilon}} |\nabla\E(\widetilde{\gamma_{n}})|_g^2\dd s \\
        &\leq \lim_{n\to\infty} \int_{I} |\nabla\E(\gamma(t_n))|_g^2\dd s = \lim_{n\to\infty} h(t_n)=0.\qedhere
    \end{align}
\end{proof}

\begin{theorem}\label{thm:euc-len-poin}
    Let $\gamma$ and $I$ be as in \cref{lem:int-lem}. Furthermore, for each $t\geq 0$, let $\phi_t\colon\D^2\to\D^2$ be any isometry. Then $\sup_{t\in[0,\infty)}\Ll_{\R^2}(\phi_t\circ\gamma(t))<\infty$.
\end{theorem}
\begin{remark}
    This theorem reveals the main advantage of working in the Poincar\'e model $\D^2$ for the hyperbolic plane: Even after applying isometries to an evolution by elastic flow, the Euclidean lengths remain uniformly bounded in time. Therefore, even if there are singularities in the sense that the hyperbolic length in $\D^2$ is unbounded in time, one can always reparametrize by constant Euclidean speed to obtain uniform control of the parametrization. This clearly fails in $\H^2$ as scalings are isometries.
\end{remark}
\begin{proof}[Proof of \cref{thm:euc-len-poin}.]
    Suppose there exists a sequence of times $t_n\nearrow \infty$ with $\Ll_{\R^2}(\phi_{t_n}\circ \gamma(t_n))\to\infty$. Denote by $\widetilde{\gamma_n}\colon [0,\Ll_{\R^2}(\phi_{t_n}\circ\gamma(t_n))]\to \D^2$ the reparametrization of $\phi_{t_n}\circ\gamma(t_n)$ by Euclidean arc-length for each $n\in\N$, that is, $|\partial_x\widetilde{\gamma_n}|\equiv 1$.

    For $R\in\N$, there is $N_0\in\N$ such that, for $n\geq N_0$, $[0,R]\subseteq [0,\Ll_{\R^2}(\phi_{t_n}\circ\gamma(t_n))]$. Applying \cref{lem:int-lem} to the sequence $(\widetilde{\gamma_n}|_{[0,R]})_{n\geq N_0}$ yields that there exists ${}^{(R)}\widetilde{\gamma_{\infty}}\colon [0,R]\to\R^2$ such that $\widetilde{\gamma_n}|_{[0,R]}\to{}^{(R)}\widetilde{\gamma_{\infty}}$ in the sense of \cref{lem:int-lem}, after passing to a subsequence also depending on $R$. Iterating this procedure for $R=1,2,\dots$, passing to further subsequences in each step, one can achieve that ${}^{(R+1)}\widetilde{\gamma_{\infty}}$ is an extension of ${}^{(R)}\widetilde{\gamma_{\infty}}$, for all $R\in\N$. By choosing a suitable diagonal sequence without relabeling, one can achieve that $\widetilde{\gamma_{n}}|_{[0,R]}\to {}^{(R)}\widetilde{\gamma_{\infty}}$ in the sense of \cref{lem:int-lem}, for each $R\in\N$.
    
    So writing $\widetilde{\gamma_{\infty}}(x)={}^{(R)}\widetilde{\gamma_{\infty}}(x)$ for any $x\in[0,\infty)$ and any $R>x$, \cref{lem:int-lem} yields that $\widetilde{\gamma_{\infty}}\in W^{1,\infty}([0,\infty),\R^2)\cap C^{0}([0,\infty),\R^2)$ and, writing $Z=\{|\widetilde{\gamma_{\infty}}|=1\}$, that $\widetilde{\gamma_{\infty}}|_{[0,\infty)\setminus Z}\in C^{\infty}$ pointwise solves $\nabla\E(\widetilde{\gamma_{\infty}})=0$ on $[0,\infty)\setminus Z$.

    In particular, for any compact interval $J\subseteq[0,\infty)$, $\widetilde{\gamma_n}|_J\to\widetilde{\gamma_{\infty}}|_J$ uniformly and, for any compact interval $J\subseteq [0,\infty)\setminus Z$, $\widetilde{\gamma_n}|_J\to\widetilde{\gamma_{\infty}}|_J$ in $C^{\infty}$.

    \textbf{Step 1.} $Z$ consists of only finitely many points. Indeed, first choose $E\in\N$ such that $8E\leq \E(\gamma_0)<8(E+1)$. Fixing any $R>0$, we then show that $Z\cap (0,R)$ has at most $E$ elements. By \cref{lem:int-lem}, one has $\widetilde{\gamma_n}|_{[0,R]}\rightharpoonup^*\widetilde{\gamma_{\infty}}|_{[0,R]}$ in $W^{1,\infty}([0,R])$ and $\widetilde{\gamma_n}|_{[0,R]}\to\widetilde{\gamma_{\infty}}|_{[0,R]}$ uniformly on $[0,R]$. Moreover, $|\partial_x\widetilde{\gamma_n}|\equiv 1$. Lastly, $\E(\widetilde{\gamma_n}|_{[0,R]})\leq \E(\gamma(t_n))\leq \E(\gamma_0)$. \cref{cor:ch-ve-poin} then yields that $Z\cap(0,R)$ consists of at most $E$ elements. Since this is true for all $R$, $Z$ itself consists of at most $E$ elements. Write $Z=\{x_1<\dots<x_m\}$ with $0\leq m\leq E$.
    
    \textbf{Step 2.} Stumbling upon a contradiction. Since $|\partial_x\widetilde{\gamma_{\infty}}|\equiv 1$ on $[0,\infty)\setminus Z$, one obtains that 
    \begin{equation}
        \begin{cases}
            \widetilde{\gamma_{\infty}}|_{(x_m,\infty)}&\text{if }Z\neq\emptyset\\
            \widetilde{\gamma_{\infty}}&\text{if }Z=\emptyset
        \end{cases}
    \end{equation}
    is an elastica with infinite Euclidean length but finite elastic energy in $\D^2$. However, this contradicts \cref{lem:el-fin-len}.
\end{proof}

\subsection{Proofs of the general convergence, quantization and blow-up results}

In this section, we present proofs for \cref{thm:quan-clo,thm:quan-cla,prop:bu-clo}.

\begin{proof}[Proof of \cref{thm:quan-clo}.]
    If $\inf_{n\in\N}\Ll_{\R^2}(\gamma_n)>0$, choose $\phi_n(z)=z$ for all $n\in\N$ and $z\in\D^2$. Otherwise, for each $n\in\N$, pick any point $p_n\in \gamma(t_n,\S^1)$ and any isometry $\phi_n\colon\D^2\to\D^2$ with $\phi_n(p_n)=(0,0)^t$. For the existence of such an isometry $\phi_n$, cf. \cite[Lemma 2.9]{muellerspener2020}. As in the statement, set $\gamma_n=\phi_n\circ\gamma(t_n)$. By \cref{thm:euc-len-poin}, one has that
    \begin{equation}
        \Ll_{\R^2}(\gamma_n) \leq L<\infty\quad\text{for all $n\in\N$}.
    \end{equation}
    
    \textbf{Claim.} One has $\inf_{n\in\N}\Ll_{\R^2}(\gamma_n)>0$. This is either already satisfied or $\phi_n$ above is chosen such that $(0,0)^t\in \gamma_n(\S^1)$ for all $n\in\N$. Then the claim follows by \cref{rem:hyp-len-lb,rem:euc-len-lb}.

    Now writing $\widetilde{\gamma_n}$ for the reparametrizations of $\gamma_n$ by constant Euclidean speed as in the statement, choosing $I'=[0,1]$, \cref{lem:int-lem} yields that there exists $\widetilde{\gamma_{\infty}}\in W^{1,\infty}(\S^1)$ with $\widetilde{\gamma_n}\to\widetilde{\gamma_{\infty}}$ uniformly. Furthermore, for $Z=\{|\widetilde{\gamma_{\infty}}|=1\}$, we have that $\widetilde{\gamma_{\infty}}|_{\S^1\setminus Z}$ is smooth and $\nabla\E(\widetilde{\gamma_{\infty}})=0$ on $\S^1\setminus Z$. 

    Moreover, using \cref{cor:ch-ve-poin}, $Z$ consists of at most finitely many elements. So write $Z=\{x_1,\dots,x_{m}\}$ for some $m\in\N_0$.

    Now, we verify \cref{eq:en-ie-clo}. To this end, consider any $1\leq j\leq m$ and $\delta>0$ such that $\{(x_j-\delta,x_j+\delta)\}_{j=1}^m$ are pairwise disjoint. By the definition of $Z$, one has that $\widetilde{\gamma_n}(x_j-\delta)$ and $\widetilde{\gamma_n}(x_j+\delta)$ converge to some point in $\D^2$, respectively. In particular, they remain bounded away from $\partial \D^2=\S^1\subseteq\R^2$. Therefore, \cref{rem:eich-grun-poin} yields that $\limsup_{n\to\infty}\E(\gamma_n|_{[x_j-\delta,x_j+\delta]})\geq 8$, for all $1\leq j\leq m$ and $\delta>0$ sufficiently small. That is, for any $\delta>0$ sufficiently small,
    \begin{equation}
        \limsup_{n\to\infty} \E(\widetilde{\gamma_n}|_{B_{\delta}(Z)})\geq 8\cdot m.
    \end{equation}
    Since $\widetilde{\gamma_n}\to\widetilde{\gamma_{\infty}}$ smoothly on $\S^1\setminus B_{\delta}(Z)$, we have that 
    $$\E(\widetilde{\gamma_n}|_{\S^1\setminus B_{\delta}(Z)}) \to \E(\widetilde{\gamma_{\infty}}|_{\S^1\setminus B_{\delta}(Z)}).$$
    Altogether, one concludes that, for any $\delta>0$ sufficiently small,
    \begin{align}
        \E(\widetilde{\gamma_{\infty}}|_{\S^1\setminus B_{\delta}(Z)}) &= \liminf_{n\to\infty} (\E(\widetilde{\gamma_n}) - \E(\widetilde{\gamma_n}|_{B_{\delta}(Z)})) \\
        &\leq \E(\gamma_0) - \limsup_{n\to\infty} \E(\widetilde{\gamma_n}|_{B_{\delta}(Z)}) \leq \E(\gamma_0) - 8\cdot m.\label{eq:com-ei} 
    \end{align}
    Letting $\delta\searrow 0$, \cref{eq:en-ie-clo} follows.

    Write now $\S^1\setminus Z = \bigcup_{i=1}^mI_i$ for open intervals $I_i$ as in the statement. As argued above, each $\widetilde{\gamma_{\infty}}|_{I_i}$ is an elastica with $\lim_{x\to\partial I_i}|\widetilde{\gamma_{\infty}}(x)|=1$, using $\partial I_i\subseteq Z$. In particular, $\Ll_{\D^2}(\widetilde{\gamma_{\infty}}|_{I_i})=\infty$ using \cref{lem:basic-facts}(a). \cref{rem:el-sm-elen} now yields that $\widetilde{\gamma_{\infty}}|_{I_i}$ either is an asymptotically geodesic elastica or a geodesic, using \cref{eq:en-ie-clo}.
\end{proof}

\begin{proof}[Proof of \cref{thm:quan-cla}.]
    As aforementioned, one essentially proceeds in the same way as in the proof of \cref{thm:quan-clo}, the only difference being that one can immediately obtain that $\Ll_{\R^2}(\widetilde{\gamma_n})$ is uniformly bounded from below --- without applying any isometries $\phi_n$. This is simply due to the boundary conditions in \cref{eq:cla-ef} and \cref{rem:hyp-len-lb,rem:euc-len-lb}. A uniform upper bound for the Euclidean lengths is again given by \cref{thm:euc-len-poin}.
    
    With this observation, one can proceed exactly as in the proof of \cref{thm:quan-clo}.
\end{proof}

\begin{proof}[Proof of \cref{prop:bu-clo}.]
    By composing $\phi_n$ with a suitable rotation, one can without loss of generality suppose that $\widetilde{\gamma_{\infty}}(x_j)=(0,-1)^t\in\partial\D^2$ and consider $u_n=\Phi\circ\widetilde{\gamma_n}|_{[x_j-\delta,x_j+\delta]}$, using \cref{eq:def-Phi}. Then there exists $\hat{y_n}\in [x_j-\delta,x_j+\delta]$ with $u_n^{(2)}(\hat{y_n})=\min_{[x_j-\delta,x_j+\delta]}u_n^{(2)}$. Define 
    \begin{equation}
        \psi_n = \Phi^{-1}\circ \left((z^{(1)},z^{(2)})^t \mapsto \frac{1}{u_n^{(2)}(\hat{y_n})} (z^{(1)}-u_n^{(1)}(\hat{y_n}),z^{(2)}) \right) \circ \Phi\colon\D^2\to\D^2.
    \end{equation}
    Then $\psi_n$ is an isometry as a composition of isometries. Furthermore, setting $\tau_n=\psi_n\circ\widetilde{\gamma_n}$, one has $\tau_n(\hat{y_n})=\Phi^{-1}((0,1)^t)=(0,0)^t$ and
    \begin{equation}\label{eq:bub-3}
        \Phi^{(2)}\circ\tau_n = \frac{1}{u_n^{(2)}(\hat{y_n})} u_n^{(2)}\geq 1.
    \end{equation}
    A brief computation using \cref{eq:def-Phi} reveals $\{\Phi^{-1}(x,1):x\in\R\} = \partial B_{\frac12}((0,\frac12)^t) \setminus \{(0,1)^t\}$ which then yields
    \begin{equation}
        \{\Phi^{-1}(x,y):x\in\R\text{ and }y\geq 1\} =\overline{B}_{\frac12}\big(\Big(0,\frac12\Big)^t\big)\setminus \{(0,1)^t\}.
    \end{equation}
    Therefore, by \eqref{eq:bub-3},
    \begin{equation}\label{eq:bub-2}
        \tau_n([x_j-\delta,x_j+\delta])\subseteq \overline{B}_{\frac12}\big(\Big(0,\frac12\Big)^t\big)\setminus \{(0,1)^t\}\subseteq\D^2.
    \end{equation}
    On the one hand, we have $u_n(x_j\pm\delta)=\Phi(\widetilde{\gamma_n}(x_j\pm\delta))\to\Phi(\widetilde{\gamma_{\infty}}(x_j\pm\delta))\in\H^2$, using $x_j\pm\delta\notin Z$. Moreover, $\Phi(\widetilde{\gamma_{\infty}}(x_j))=\Phi((0,-1)^t)=(0,0)^t$ which yields $u_n^{(2)}(\hat{y_n})\to 0$ for $n\to\infty$. Thus,
    \begin{equation}
        \Phi^{(2)}(\tau_n(x_j\pm\delta)) = \frac{1}{u_n^{(2)}(\hat{y_n})} u_n^{(2)}(x_j\pm\delta)\to\infty\quad\text{for $n\to\infty$.}
    \end{equation}
    Consequently, 
    \begin{equation}\label{eq:bub-1}
        \widetilde{\tau_n}(\pm1)=\tau_n(x_j\pm\delta)\to (0,1)^t \quad\text{for $n\to\infty$},
    \end{equation} using \cref{eq:inv-Phi}. With $\tau_n(\hat{y_n})=(0,0)^t$, we obtain $\inf_{n\in\N}\Ll_{\R^2}(\widetilde{\tau_n})=\inf_{n\in\N}\Ll_{\R^2}(\tau_n)>0$.

    Using \cref{lem:int-lem,thm:euc-len-poin}, there exists $\widetilde{\tau_{\infty}}\colon[-1,1]\to\overline{\D^2}$ with $\widetilde{\tau_n}\to\widetilde{\tau_{\infty}}$ uniformly on $[-1,1]$ and smoothly on compact subsets of $[-1,1]\setminus Z'$ where $Z'=\{|\widetilde{\tau_{\infty}}|=1\}$ and $\pm1\in Z'$ by \cref{eq:bub-1}. Moreover, $\widetilde{\tau_n}\rightharpoonup^*\widetilde{\tau_{\infty}}$ in $W^{1,\infty}([-1,1])$ and $\nabla\E(\widetilde{\tau_{\infty}})=0$ on $[-1,1]\setminus Z'$. Then \cref{cor:ch-ve-poin} implies that $Z'$ consists of only finitely many elements. Write $Z'=\{y_0<\dots<y_{n_j}\}$ for $n_j\in\N$ where $y_0=-1$ and $y_{n_j}=1$ as explained above. 
    
    Using \cref{eq:bub-2} and the uniform convergence $\widetilde{\tau_n}\to\widetilde{\tau_{\infty}}$, it follows that 
    $$
        \widetilde{\tau_{\infty}}(y_j)\in \partial\D^2\cap \overline{B}_{\frac12}\big(\Big(0,\frac12\Big)^t\big)=\{(0,1)^t\}.
    $$ 
    Altogether, $\widetilde{\tau_{\infty}}|_{(y_{i-1},y_i)}$ is an elastica with both endpoints equal to $(0,1)^t\in\partial\D^2$ and finite elastic energy. \cref{rem:el-sm-elen} yields that each $\widetilde{\tau_{\infty}}|_{(y_{i-1},y_i)}$ necessarily parametrizes an asymptotically geodesic elastica. Thus,
    \begin{equation}
        \limsup_{n\to\infty} \E(\widetilde{\tau_n}|_{(y_{i-1},y_i)}) \geq \E(\widetilde{\tau_{\infty}}|_{(y_{i-1},y_i)}) = 8.
    \end{equation}
    In particular, $\limsup_{n\to\infty}\E(\widetilde{\gamma_n}|_{(x_j-\delta,x_j+\delta)})=\limsup_{n\to\infty}\E(\widetilde{\tau_n})\geq 8n_j$ which yields \cref{eq:impr-ei} with a similar computation as in \cref{eq:com-ei}.
\end{proof}

\section{Singular behavior of elastic flow of symmetric figure-eights}\label{sec:fig-eights}

This section is dedicated to proving \cref{thm:sf8-conv}. As in the introduction, we make the identification $\S^1=\nicefrac{[-2,2]\ }{\sim}$ where we define $\sim$ by $x\sim y\vcentcolon\iff (x=y) \lor (\{x,y\}=\{-2,2\})$. Firstly, we specify the notion of symmetric figure-eights.

\begin{definition}
    An immersion $\gamma\colon\S^1\to\D^2$ is called \emph{symmetric figure-eight} if its winding number vanishes and
    \begin{itemize}
        \item[(S1)] $\gamma(-x)=-\gamma(x)$ for all $x\in\S^1$ and 
        \item[(S2)] $\gamma(\pm1-x)=\begin{pmatrix}
            -1&0\\0&1
        \end{pmatrix} \gamma(\pm1+x)$ for all $x\in[-1,1]$. 
    \end{itemize}
\end{definition}

If $\gamma$ satisfies (S1), one obtains $\gamma(0)=\gamma(\pm2)=0$. Furthermore,
\begin{equation}\label{eq:s1-and-curv}
    \curv(-x)=\nabla_s(\partial_s\gamma(\cdot))(-x)\underset{(S1)}{=} -\nabla_s(\partial_s(\gamma(-\ \cdot)))(-x) = -(\nabla_s\partial_s\gamma)(x) = -\curv(x)
\end{equation}
so that (S1) also implies $\curv(0)=\curv(\pm2)=0$.
\begin{remark}[Existence of symmetric figure-eights]
    In fact, there is an --- in the context of the elastic flow --- particularly noteworthy class of such symmetric figure-eights. Namely, the so-called $\lambda$-figure-eights are already used in \cite{muellerspener2020,dallacquamullerschatzlespener2020} to prove that the elastic flow (and the Willmore flow of tori of revolution) can develop singularities above a certain energy threshold below which convergence can be proved.

    A $\lambda$-figure-eight is a closed, $\lambda$-constrained elastica (cf. \cref{def:elastica}) with vanishing winding number. It is shown in \cite[Proposition 6.2 and Corollary 6.4]{muellerspener2020} that, for any $\lambda\in(0,64/\pi^2-2)$, there exists a $\lambda$-figure-eight $\gamma_{\lambda}$ and $\E(\gamma_{\lambda})\searrow 16$ for $\lambda\searrow 0$. Moreover, one can choose $\gamma_{\lambda}$ such that both (S1) and (S2) are satisfied, cf. \cite[Lemma 5.15 and Remark 5.21]{schlierf2023}. 

    For an illustration of $\lambda$-figure-eights, see \Cref{fig:lamb-fig-eights}. Strikingly, one realizes that the illustrated sequence $(\gamma_{\lambda})_{\lambda\in (0,64/\pi^2-2)}$ has the same convergence behavior for $\lambda\searrow 0$ as the elastic flow in \cref{thm:sf8-conv} starting in any fixed $\lambda$-figure-eight for $t\to\infty$.
\end{remark}

An important step in proving \cref{thm:sf8-conv} is the following proposition which shows that the center $(0,0)^t\in\D^2$ is a stationary point along the evolution of the elastic flow of symmetric figure-eights.

\begin{proposition}\label{prop:S1-for-flow}
    Let $\gamma_0\colon\S^1\to\D^2$ be an immersion satisfying (S1) and let $\gamma\colon[0,\infty)\times\S^1\to\D^2$ be the solution to the elastic flow \cref{eq:clo-ef}. Then $\gamma(t,-x)=-\gamma(t,x)$ for all $t\geq 0$ and $x\in\S^1$. In particular,
    \begin{equation}\label{eq:cut-nat-bdr}
        \gamma(t,0)=\gamma(t,\pm2) = (0,0)^t,\quad\curv(t,0)=\curv(t,\pm2)=(0,0)^t 
    \end{equation}
    for all $t\geq 0$. Moreover, if $\gamma_0\colon\S^1\to\S^1$ is an immersion satisfying (S2), then also $\gamma(t)$ satisfies (S2) for all $t\geq 0$.
\end{proposition}
\begin{remark}
    Equation \cref{eq:cut-nat-bdr} yields that $\gamma|_{[0,\infty)\times [0,2]}$ and $\gamma|_{[0,\infty)\times[-2,0]}$ solve the elastic flow with initial data $\gamma_0|_{[0,2]}$ and $\gamma_0|_{[-2,0]}$ with so-called \emph{natural boundary conditions}, respectively.
\end{remark}
The proof of \cref{prop:S1-for-flow} is based on the observation that the elastic flow preserves all kinds of symmetries which can be described by isometries of $\D^2$. More precisely:
\begin{lemma}[Symmetry preservation]\label{lem:sym-pres}
    Let $F\colon\D^2\to\D^2$ be an isometry of $\D^2$ and $\varphi\colon\S^1\to\S^1$ a diffeomorphism. If $\gamma_0\colon\S^1\to\D^2$ is an immersion with
    \begin{equation}
        \gamma_0(x)=F(\gamma_0(\varphi(x)))\quad\text{for all $x\in\S^1$},
    \end{equation}
    then, for the solution $\gamma\colon[0,\infty)\times\S^1\to\D^2$ to \cref{eq:clo-ef}, one has
    \begin{equation}\label{eq:sym-pres}
        \gamma(t,x) = F(\gamma(t,\varphi(x)))\quad\text{for all $x\in\S^1$ and $t\geq 0$}.
    \end{equation}
\end{lemma}
\begin{proof}
    Define $\alpha\colon[0,\infty)\times \S^1\to\D^2$ by $\alpha(t,x)=F(\gamma(t,\varphi(x)))$. Then 
    \begin{align}
        \curv_{\alpha}(t,x)&=\nabla_{s} \partial_s\alpha(t,x) =  \nabla_s dF_{\gamma(t,\varphi(x))}(\partial_s (\gamma(t,\varphi(x)))) \\
        &= dF_{\gamma(t,\varphi(x))}(\nabla_s\partial_s(\gamma(t,\varphi(x)))) = dF_{\gamma(t,\varphi(x))}\curv_{\gamma}(t,\varphi(x))\label{eq:sym-curv},
    \end{align}
    using that $F$ is an isometry and that therefore $dF$ commutes with the covariant derivative. This immediately yields that 
    \begin{equation}
        \nabla\E(\alpha(t))(x) = dF_{\gamma(t,\varphi(x))}\nabla\E(\gamma(t)) (\varphi(x)).
    \end{equation}
    Therefore,
    \begin{align}
        \partial_t\alpha(t,x)&=dF_{\gamma(t,\varphi(x))}(\partial_t\gamma)(t,\varphi(x)) = -dF_{\gamma(t,\varphi(x))}(\nabla\E(\gamma(t))(\varphi(x)))\\
        &= -\nabla\E(\alpha(t))(x).
    \end{align}
    Moreover, $\alpha(0,x)=F(\gamma_0(\varphi(x)))=\gamma_0(x)$ by assumption. Therefore, $\alpha$ also solves \cref{eq:clo-ef}. By geometric uniqueness of \cref{eq:clo-ef} (cf. \cite[Theorem 3.1]{dallacquaspener2017}), we obtain that $\gamma=\alpha$. Therefore, the claim follows.
\end{proof}


\begin{proof}[Proof of \cref{prop:S1-for-flow}.]
    Choosing $\varphi(x)=-x$ and $F(p)=-p$, we have $\gamma(t,-x)=-\gamma(t,x)$ for all $t\geq 0$ and $x\in\S^1$ by \cref{lem:sym-pres}, i.e. (S1) for $\gamma(t)$ for each $t\geq 0$. Then \cref{eq:cut-nat-bdr} follows by \cref{eq:s1-and-curv}.

    Regarding the preservation of (S2), define $\varphi\colon\S^1\to\S^1$ by
    \begin{equation}\label{eq:sym-s2-varphi}
        \varphi(x)=\begin{cases}
            -2-x&x\in[-2,0]\\
            2-x&x\in[0,2]
        \end{cases}
    \end{equation}
    where $\varphi$ is well-defined at $x=0$, since $-2\sim 2$. Moreover, $\varphi(-2)=0=\varphi(2)$. Furthermore, $F(p)=(-p^{(1)},p^{(2)})^t$ is an isometry of $\D^2$. Preservation of (S2) then follows from \cref{lem:sym-pres} applied with these choices of $\varphi$ and $F$.
\end{proof}

\begin{remark}\label{rem:sym-rep}
    Another important observation relates to how such symmetries behave with respect to reparametrizations, e.g. by constant Euclidean speed. To this end, denote by $I$ either a compact interval or $\S^1$. Let $\varphi\colon I\to I$ be a diffeomorphism and $F\colon\R^2\to\R^2$. Consider an immersion $\gamma\colon I\to\R^2$ with $\gamma(x)=F(\gamma(\varphi(x)))\quad\text{for all $x\in I$}$. 

    Let $\psi\colon I\to I$ be a diffeomorphism and $\widetilde{\gamma} = \gamma\circ\psi$. If 
    \begin{equation}\label{eq:sym-rep-2}
        \varphi(\psi(x)) = \psi(\varphi(x)) \quad\text{for all $x\in I$},
    \end{equation}
    then also $\widetilde{\gamma}(x)=F(\widetilde{\gamma}(\varphi(x)))\quad\text{for all $x\in I$}$. 
\end{remark}

\begin{proof}[Proof of \cref{thm:sf8-conv}.]
    Consider a sequence $t_n\nearrow \infty$. \cref{prop:S1-for-flow} yields that $\gamma(t_n)$ satisfies (S1). Writing $\psi_n\colon\S^1\to\S^1$ for the inverse of $y\mapsto \frac{4}{\Ll_{\R^2}(\gamma(t_n))}\int_0^y |\partial_x\gamma(t_n)|\dd x$, then $\widetilde{\gamma_n}=\gamma(t_n)\circ\psi_n$ is a reparametrization of $\gamma(t_n)$ by constant Euclidean speed. Moreover, as $-\psi_n^{-1}(y)=\psi_n^{-1}(-y)$, \cref{eq:sym-rep-2} is satisfied with $\varphi(x)=-x$. Thus, using \cref{rem:sym-rep}, also $\widetilde{\gamma_n}$ satisfies (S1). Thus, \cref{eq:cut-nat-bdr} remains true for each $\widetilde{\gamma_n}$. \cref{rem:hyp-len-lb,rem:euc-len-lb} yield $\inf_{n\in\N}\Ll_{\R^2}(\widetilde{\gamma_n})>0$. 

    Moreover, by \cref{thm:euc-len-poin}, one also has $\sup_{n\in\N}\Ll_{\R^2}(\widetilde{\gamma_n})<\infty$. Altogether, \cref{lem:int-lem} yields that, after passing to a subsequence, there exists $\widetilde{\gamma_{\infty}}\in W^{1,\infty}(\S^1)$ with $\widetilde{\gamma_n}\rightharpoonup^*\widetilde{\gamma_{\infty}}$ in $W^{1,\infty}(\S^1)$ and $\widetilde{\gamma_n}\to \widetilde{\gamma_{\infty}}$ uniformly. Furthermore, for $Z=\{|\widetilde{\gamma_{\infty}}|=1\}$, we have smooth convergence on compact sets in $\S^1\setminus Z$ and $\widetilde{\gamma_{\infty}}$ is an elastica in $\S^1\setminus Z$. In particular, $\widetilde{\gamma_{\infty}}$ also satisfies (S1) due to uniform convergence.

    Using \cref{prop:S1-for-flow}, each $\gamma(t_n)$ satisfies (S2). We now argue that also each $\widetilde{\gamma_n}$ satisfies (S2). Indeed, with $\varphi$ as in \cref{eq:sym-s2-varphi}, we have for $y\in[0,2]$
    \begin{align}
        \psi_n^{-1}(\varphi(y)) &= \psi_n^{-1}(2-y) = \frac{4}{\Ll_{\R^2}(\gamma(t_n))}\int_0^{2-y} |\partial_x\gamma(t_n)|\dd x \\
        &= \frac{4}{\Ll_{\R^2}(\gamma(t_n))}\int_0^{2} |\partial_x\gamma(t_n)|\dd x - \psi_n^{-1}(y) = 2-\psi_n^{-1}(y) = \varphi(\psi_n^{-1}(y)),
    \end{align}
    using (S2) for $\gamma(t_n)$. A similar computation works for $y\in[-2,0]$ so that \cref{eq:sym-rep-2} is satisfied. \cref{rem:sym-rep} yields that each $\widetilde{\gamma_n}$ also satisfies (S2). By uniform convergence, $\widetilde{\gamma_{\infty}}$ also satisfies (S2). 

    \cref{cor:ch-ve-poin} yields that $Z=\{x_1,\dots,x_m\}$ where $m\leq 3$, using $\E(\gamma_0)<32=4\cdot 8$. 
    By (S1), $\widetilde{\gamma_{\infty}}(-x)=-\widetilde{\gamma_{\infty}}(x)$ for all $x\in\S^1$ which yields $0,\pm2\notin Z$ as in \cref{eq:cut-nat-bdr}. Thus, the cardinality of $Z$ is necessarily an even number. That is, $m\leq 3$ yields $m=0$ or $m=2$.

    If however $m=0$, $Z=\emptyset$ and $\widetilde{\gamma_n}\to\widetilde{\gamma_{\infty}}$ smoothly on all of $\S^1$. Since the winding number is preserved along the flow, one obtains from the smooth convergence on $\S^1$ that $\widetilde{\gamma_{\infty}}$ is a closed elastica with winding number $0$. But such elastica do not exist by \cite[Corollary 5.8]{muellerspener2020}. Therefore, $m=2$.

    By (S1), $x_1=-x_2$ so that $Z=\{x_1,-x_1\}$ where we suppose without loss of generality that $x_1\in (0,2)$, using again (S1) as in \cref{eq:cut-nat-bdr}. Due to \cref{eq:en-ie-clo}, we have $\E(\widetilde{\gamma_{\infty}}|_{\S^1\setminus\{-x_1,x_1\}})\leq \E(\gamma_0)-16 < 16$. 

    Using also \cref{lem:basic-facts}(a), $\widetilde{\gamma_{\infty}}|_{(-x_1,x_1)}$ is an elastica with infinite hyperbolic length and finite elastic energy where the end-points satisfy
    \begin{equation}\label{eq:ep-geo}
        \widetilde{\gamma_{\infty}}(-x_1)=-\widetilde{\gamma_{\infty}}(x_1)\in\partial\D^2\subseteq\R^2.
    \end{equation}
    Using \cref{rem:el-sm-elen}, this yields that $\widetilde{\gamma_{\infty}}|_{(-x_1,x_1)}$ is a geodesic since both endpoints on $\partial\D^2$ of a globally defined, asymptotically geodesic elastica coincide, cf. \Cref{fig:elastica}. Using \cref{eq:ep-geo}, it follows that $\widetilde{\gamma_{\infty}}|_{(-x_1,x_1)}$ parametrizes the straight line connecting $-\widetilde{\gamma}(-x_1)$ and $\widetilde{\gamma}(-x_1)$.  Similarly, one obtains that also $\widetilde{\gamma_{\infty}}|_{\S^1\setminus[-x_1,x_1]}$ parametrizes the straight line connecting $-\widetilde{\gamma}(-x_1)$ and $\widetilde{\gamma}(-x_1)$. Since $|\partial_x\widetilde{\gamma_{\infty}}|$ is constant on $\S^1\setminus\{- x_1,x_1\}$, it follows that $x_1=1$ and therefore $|\partial_x\widetilde{\gamma_{\infty}}|\equiv 1$ on $\S^1\setminus\{-1,1\}$.

    Since $\widetilde{\gamma_{\infty}}$ is continuous and satisfies (S2), we necessarily have that 
    \begin{equation}
        \widetilde{\gamma_{\infty}}(\pm1)=\begin{pmatrix}
            -1&0\\0&1
        \end{pmatrix}\widetilde{\gamma_{\infty}}(\pm1),
    \end{equation}
    so that $\{\widetilde{\gamma_{\infty}}(-1),\widetilde{\gamma_{\infty}}(1)\}=\{-(0,1)^t,(0,1)^t\}$, using also $Z=\{-1,1\}$ and \cref{eq:ep-geo}. This yields the claim.
\end{proof}

\section{Singular behavior of elastic flow of vertically clamped symmetric curves}\label{sec:vert-cla}

\begin{definition}
    We say that an immersion $\gamma\colon[-1,1]\to\D^2$ satisfies (S2'), if 
    \begin{equation}
        \gamma(-x)=\begin{pmatrix}
        -1&0\\0&1
        \end{pmatrix}\gamma(x)\quad\text{for all $x\in[-1,1]$}.
    \end{equation}
\end{definition}

\begin{remark}[On symmetric, ``vertically clamped'' curves]
    Consider $\gamma\colon[-1,1]\to\D^2$. By calling $\gamma$ ``vertically clamped'' in this section's title, we mean that $\gamma$ satisfies 
    \begin{equation}\label{eq:sing-bdry}
        \gamma(\pm1)=(0,0)^t\quad\text{and}\quad \partial_s\gamma(\pm1)=\pm(0,1)^t.
    \end{equation}
    Furthermore, in this section, $\gamma$ is called symmetric, if it satisfies (S2'). More explicitly, in \cite[Section 5]{schlierf2023}, a family $(\gamma_n)_{n\in\N}$ of symmetric, vertically clamped curves is constructed with $\E(\gamma_n)\searrow 8$. Furthermore, using the Li-Yau inequality in \cite[Theorem 6.1]{schlierf2023}, any such vertically clamped curve $\gamma$ satisfies $\E(\gamma)>8$ due to the prescribed self-intersection $\gamma(-1)=\gamma(1)$. 

    In \cite{schlierf2023}, the associated surfaces of revolution of the profile curves in $\H^2$ of such symmetric, vertically clamped curves $\gamma_n$ are used as initial data for the Willmore flow with Dirichlet boundary data to show that the flow can develop singularities. As we show below, something similar happens also in the case of the elastic flow --- however, here we can exactly describe the convergence behavior and the limit.
\end{remark}

\begin{theorem}
    Consider an immersion $\gamma_0\colon[-1,1]\to\D^2$ satisfying the boundary conditions \eqref{eq:sing-bdry}, (S2') and $8<\E(\gamma_0)<16$. If $\widetilde{\gamma(t)}$ is the reparametrization of $\gamma(t)$ by constant Euclidean speed, then $\widetilde{\gamma(t)}\to\widetilde{\gamma_{\infty}}$ uniformly and smoothly on compact subsets of $[-1,1]\setminus\{0\}$ where 
    \begin{equation}
        \widetilde{\gamma_{\infty}}(x)=\begin{cases}
            (0,-1-x)&x\in[-1,0]\\
            (0,-1+x)&x\in[0,1].
        \end{cases}
    \end{equation}
\end{theorem}
\begin{proof}
    Let $t_n\nearrow \infty$ and write $\widetilde{\gamma_n}=\widetilde{\gamma(t_n)}$. Using \cref{thm:quan-cla}, there exists $\widetilde{\gamma_{\infty}}$ with $\widetilde{\gamma_n}\to\widetilde{\gamma_{\infty}}$ uniformly on $[-1,1]$ and smoothly on compact sets in $[-1,1]\setminus Z$ where $Z=\{|\widetilde{\gamma_{\infty}}|=1\}$ either satisfies $Z=\emptyset$ or $Z=\{x_1\}$, using $\E(\gamma_0)<16$. Furthermore, $\nabla\E(\widetilde{\gamma_{\infty}})=0$ on $[-1,1]\setminus Z$.

    By \cref{lem:sym-pres}, each $\gamma(t_n)$ satisfies (S2') and, using \cref{rem:sym-rep}, so does each $\widetilde{\gamma_n}$. By the uniform convergence, also $\widetilde{\gamma_{\infty}}$ satisfies (S2').    

    If $Z=\emptyset$, then $\widetilde{\gamma_{\infty}}$ is an elastica with (S2') and the boundary conditions of \cref{eq:sing-bdry}. This contradicts the arguments of \cite[Remark 5.23]{schlierf2023} showing that no such elastica exists. That is, $Z=\{x_1\}$ for some $x_1\in[-1,1]$. Using (S2'), $|\widetilde{\gamma_{\infty}}(-x_1)|=|\widetilde{\gamma_{\infty}}(x_1)|=1$. So $-x_1\in Z=\{x_1\}$, i.e. $x_1=0$. Moreover, note that (S2') and $|\widetilde{\gamma_{\infty}}(0)|=1$ yield $\widetilde{\gamma_{\infty}}(0)=(0,\pm1)^t$. Lastly, using that the segments $\widetilde{\gamma_{\infty}}|_{[-1,0)}$ and $\widetilde{\gamma_{\infty}}|_{(0,1]}$ are isometric in $\D^2$, we have $\E(\widetilde{\gamma_{\infty}}|_{[-1,0)})<4$ by \cref{eq:en-ie-cla}. By \cref{rem:el-sm-elen}, $\widetilde{\gamma_{\infty}}|_{[-1,0)}$ either is a segment of a geodesic or of an asymptotically geodesic elastica. If $\widetilde{\gamma_{\infty}}|_{[-1,0)}$ is a segment of a geodesic, then $\widetilde{\gamma_{\infty}}(0)=(0,\pm1)^t$ and \cref{eq:sing-bdry} yield the claim.

    If however $\widetilde{\gamma_{\infty}}|_{[-1,0)}$ was a segment of an asymptotically geodesic elastica, $\widetilde{\gamma_{\infty}}(0)=(0,\pm1)^t$ and $\widetilde{\gamma_{\infty}}(\pm1)=(0,0)^t$, $\partial_s\widetilde{\gamma_{\infty}}(\pm1)=(0,\pm1)^t$ would contradict \cref{rem:on-as-geo-el} applied to $\Phi\circ(\mp\widetilde{\gamma_{\infty}})$, using \cref{eq:def-Phi}.
\end{proof}

\appendix

\section{Technical and computational proofs}\label{app:el}

\begin{proof}[Proof of \cref{lem:basic-facts}.]
    For (a), by the fundamental theorem of calculus and the Cauchy-Schwarz inequality, using $|\gamma(c)|^2\to 1$ for $c\nearrow b$ and $|\gamma|\leq 1$,
    \begin{align}
        \infty &= \log(1-|\gamma(a)|^2) - \lim_{c\nearrow b} \log(1-|\gamma(c)|^2) \leq \lim_{c\nearrow b}\int_a^{c} \frac{2|\langle \gamma,\partial_x\gamma\rangle|}{1-|\gamma|^2}\dd x\\
        &\leq \int_a^b \frac{2|\partial_x\gamma|}{1-|\gamma|^2}\dd x = \int_a^b |\partial_x\gamma|_g\dd x = \Ll_{\D^2}(\gamma).
    \end{align}
    For (b), first note $0\leq |\gamma|^2<1$ so that $\frac{2}{1-|\gamma|^2}\geq 2$. Hence,
    \begin{equation}
        \Ll_{\D^2}(\gamma) = \int_I \frac{2}{1-|\gamma|^2}|\partial_x\gamma|\dd x \geq 2\int_I |\partial_x\gamma|\dd x=2\Ll_{\R^2}(\gamma)
    \end{equation}
    and the claim follows.
\end{proof}

\begin{proof}[Proof of \cref{cor:bu-len}]
    Firstly, by restricting each $u_n$ to a sub-interval of $I$ and then reparametrizing the restriction to a new curve in $I$, one can justify that w.l.o.g $\ell\leq \Ll_{\R^2}(u_n)\leq 2\ell=L$ for all $n\in\N$. 

    First, if $\int_I |\partial_xu_n^{(1)}|\dd x\not\to 0$, we can simply apply \cref{cor:bu-elen} to obtain the claim in this case. Therefore, without loss of generality suppose that $\int_I |\partial_xu_n^{(1)}|\dd x\to 0$ for $n\to\infty$. We further assume that $I=[0,1]$ and $|\partial_xu_n|\equiv \Ll_{\R^2}(u_n)=\vcentcolon \ell_n$. 

    Let $0<\varepsilon<\min\{\frac{1}{32}\ell,\frac{1}{64},\ell^2\}$ and let $n$ be large enough such that $\int_I|\partial_xu_n^{(1)}|\dd x < \varepsilon$ and $u_n^{(2)}< \varepsilon$. In the following, we fix some such sufficiently large $n\in\N$. 
    First, we argue that
    \begin{equation}\label{eq:bu-len-101}
        |\{\partial_xu_n^{(2)}\neq 0\}|\geq\frac14.
    \end{equation}
    Indeed, using $\ell\leq |\partial_xu_n|\leq L=2\ell$ and Fubini, we have
    \begin{align}
        \ell&\leq\int_I|\partial_xu_n|\dd x \leq \varepsilon + \int_I |\partial_xu_n^{(2)}|\dd x \\
        &<\frac\ell2 + \int_0^{2\ell} |\{|\partial_xu_n^{(2)}|>\tau\}|\dd \tau \leq \frac\ell2 + 2\ell \cdot |\{\partial_xu_n^{(2)}\neq 0\}| 
    \end{align}
    so that \cref{eq:bu-len-101} follows. Now denote by $(I_j)_{j=1}^{N_{\varepsilon}}$ the closures of the countably many connected components of the open set $\{\partial_xu_n^{(2)}\neq 0\}$ where $N_{\varepsilon}\in\N\cup\{\infty\}$. By \cref{eq:bu-len-101}, 
    \begin{equation}\label{eq:bu-len-102}
        \sum_{j=1}^{N_{\varepsilon}} |I_j| = \sum_{j=1}^{N_{\varepsilon}} |I_j^{\circ}| = |\{\partial_xu_n^{(2)}\neq 0\}|\geq\frac14.
    \end{equation}
    If $I_j=[a_j,b_j]$ where we label the intervals such that, without loss of generality, $I_j\subseteq (0,1)$ for $j\geq 3$, then $\partial_xu_n^{(2)}(a_j)=\partial_xu_n^{(2)}(b_j)=0$ for $j\geq 3$. In particular, $\partial_xu_n^{(2)}$ does not change sign on $I_j$ for $j\geq 3$. Since $0<u_n^{(2)}\leq \varepsilon$, one has $\int_{I_j} |\partial_xu_n^{(2)}|\dd x = | \int_{I_j} \partial_xu_n^{(2)}\dd x | < \varepsilon$, using $u_n^{(2)}>0$, and thus
    \begin{equation}\label{eq:bu-len-1}
        |I_j| = \frac{1}{\ell_n} \int_{I_j} |\partial_xu_n|\dd x \leq \frac{\varepsilon + \varepsilon}{\ell_n} \leq 2\frac{\varepsilon}{\ell}.
    \end{equation}
    Note that \cref{eq:bu-len-102,eq:bu-len-1} immediately result in $\frac14 \leq 2\frac\varepsilon\ell N_{\varepsilon}$, i.e. $N_{\varepsilon}\geq \frac{\ell}{8\varepsilon}\geq 4$. Moreover, one obtains 
    \begin{equation}\label{eq:bu-len-001}        
        \sum_{j=3}^{N_{\varepsilon}}|I_j|\geq \frac14-4\frac\varepsilon\ell\geq \frac{1}{8}>\sqrt{\varepsilon}.
    \end{equation}
    So we can choose $3\leq J_1\leq N_{\varepsilon}$ minimal such that 
    \begin{equation}\label{eq:bu-len-100}
        \sum_{j=3}^{J_1}|I_j|\geq \sqrt{\varepsilon}.
    \end{equation}
    Since $\int_I|\partial_xu_n^{(1)}|\dd x < \varepsilon$, there exist $j_1\in\{3,\dots,J_1\}$ and $\xi_1\in J_{j_1}$ with 
    \begin{equation}
        |\partial_xu_n^{(1)}(\xi_1)| < \frac{\varepsilon}{\sqrt{\varepsilon}} = \sqrt{\varepsilon}. 
    \end{equation}
    In particular, using $(\partial_xu_n^{(2)})^2 = \ell_n^2-(\partial_xu_n^{(1)})^2 \geq \ell^2-(\partial_xu_n^{(1)})^2$, $|\partial_xu_n^{(2)}(\xi_1)|\geq\sqrt{\ell^2-\varepsilon}$. 
    
    We now distinguish two cases. Firstly, if $\partial_xu_n^{(2)}\geq 0$ on $I_{j_1}$, then by \cref{eq:br-gr},
    \begin{equation}
        \E(u_n|_{[a_{j_1},\xi_1]}) \geq 4 \frac{\partial_xu_n^{(2)}(\xi_1)-\partial_xu_n^2(a_{j_1})}{\ell_n} = 4 \frac{\partial_xu_n^{(2)}(\xi_1)}{\ell_n} \geq \frac{4\sqrt{\ell^2-\varepsilon}}{L}.
    \end{equation}
    If however $\partial_xu_n^{(2)}\leq 0$ on $I_{j_1}$, then again by \cref{eq:br-gr},
    \begin{equation}
        \E(u_n|_{[\xi_1,b_{j_1}]}) \geq 4 \frac{\partial_xu_n^{(2)}(b_{j_1})-\partial_xu_n^2(\xi_1)}{\ell_n} = 4 \frac{-\partial_xu_n^{(2)}(\xi_1)}{\ell_n} \geq \frac{4\sqrt{\ell^2-\varepsilon}}{L}.
    \end{equation}
    Altogether, we have that
    \begin{equation}
        \E(u_n|_{\bigcup_{j=3}^{J_1}I_j}) \geq \frac{4\sqrt{\ell^2-\varepsilon}}{L}.
    \end{equation}
    Moreover, note that $\sum_{j=3}^{J_1}|I_j|\leq \sqrt{\varepsilon} + 2\frac\varepsilon\ell\leq 3\sqrt{\varepsilon}$, using \cref{eq:bu-len-1,eq:bu-len-100} and the minimality of $J_1$. 

    Inductively repeating the above arguments, there exist $J_1<J_2<\dots<J_{M_{\varepsilon}}\leq N_{\varepsilon}$ such that, writing $J_0=2$,
    \begin{equation}
        \E(u_n|_{\bigcup_{j=J_{i-1}+1}^{J_i}I_j}) \geq \frac{4\sqrt{\ell^2-\varepsilon}}{L}\quad\text{for all $i=1,\dots,M_{\varepsilon}$}
    \end{equation}
    and where $M_{\varepsilon} \geq \frac{\sum_{j=3}^{N_{\varepsilon}}|I_j|}{3\sqrt{\varepsilon}}\geq \frac{1}{24\sqrt{\varepsilon}}$, using that, by \cref{eq:bu-len-001}, $\sum_{j=3}^{N_{\varepsilon}}|I_j|\geq \frac18$, and the fact that $\sum_{j=J_{i-1}+1}^{J_i}|I_j|\leq 3\sqrt{\varepsilon}$ which holds by construction. This proves the claim.
\end{proof}

\begin{proof}[Proof of \cref{rem:on-as-geo-el}]
    Firstly, all asymptotically geodesic elastica are isometric to one another by the fundamental theorem of curve theory, using \cite[Lemma 2.9]{muellerspener2020} and \cref{prop:2.8}. A parametrization by arc-length of an asymptotically geodesic elastica $u\colon\R\to\H^2$ is given by
    \begin{equation}\label{eq:as-geo-el}
        u(x) = \frac{1}{x^2+\cosh^2(x)} (x,\cosh(x))^t.
    \end{equation}
    Suppose that there exists a segment $\sigma\colon[0,1)\to\H^2$ of an asymptotically geodesic elastica with $\sigma(0)=(0,1)^t$ and $\partial_s\sigma(0)=(0,\pm1)^t$ and $\lim_{x\nearrow 1}\sigma(x)=(0,0)^t$. Then consider the geodesic $\tau\colon [0,1)\to\H^2$ given by $\tau(x)=(1-x)(0,1)^t$. Note that $\tau(0)=(0,1)^t$ and $\partial_s\tau(0)=(0,-1)^t$ and $\lim_{x\nearrow 1}\tau(x)=(0,0)^t$. As aforementioned, there exists an isometry $F\colon\H^2\to\H^2$ such that $F\circ\sigma$ is a reparametrization of a segment of $u$. Thus, $\hat{\tau}=F\circ\tau$ is a geodesic with $\lim_{x\nearrow 1}\hat{\tau}(x)=(0,0)^t$ and which touches $u$ tangentially at $\hat{\tau}(0)$. By a straight-forward computation, using \cref{eq:as-geo-el} and the classification of geodesics in $\H^2$, one obtains that this is impossible.

    Indeed, without loss of generality, up to reparametrization, $\hat{\tau}$ is a segment of $v\colon(0,\pi)\to\H^2$ given by $v(x)=h(\cos(x)+1,\sin(x))$ for some $h>0$. One then computes that the only $x_u\in\R$, $x_v\in (0,\pi)$ with $u(x_u)=v(x_v)$ are given by 
    \begin{equation}
        x_u=\frac{1}{2h}\quad\text{and}\quad x_v=2 \arctan\left(2 h \cosh \left(\frac{1}{2 h}\right)\right).
    \end{equation}
    Moreover, one obtains that
    \begin{equation}
        \det\big(\partial_xu(x_u) \ | \ \partial_xv(x_v)\big) = -\frac{4 h^3}{2 h^2+2 h^2 \cosh \left(\frac{1}{h}\right)+1} \neq 0
    \end{equation}
    which yields that $u$ and $v$ intersect transversally, i.e. their tangent vectors are not collinear at their intersection.
\end{proof}

\begin{proof}[Proof of \cref{lem:el-fin-len}.]
    Since, for any immersion $\gamma\colon I\to\D^2$ of an interval $I$ into $\D^2$, $\Ll_{\D^2}(\gamma)\geq 2\Ll_{\R^2}(\gamma)$ by \cref{lem:basic-facts}(b), \cref{rem:el-sm-elen} yields that, if there are elastica with finite elastic energy but infinite Euclidean length, they are either geodesics or asymptotically geodesic. However, it is well known that all geodesics in $\D^2$ have finite Euclidean length. Moreover, all asymptotically geodesic elastica are isometric to one another by the fundamental theorem of curve theory, using \cite[Lemma 2.9]{muellerspener2020} and \cref{prop:2.8}.

    One parametrization of an asymptotically geodesic elastica by hyperbolic arc-length in $\D^2$ is given by $\gamma=\Phi^{-1}\circ u\colon\R\to\D^2$, with $\Phi$ and $u$ as in \cref{eq:def-Phi,eq:as-geo-el}, that is
    \begin{equation}\label{eq:param-as-geo}
        \gamma(x)=\frac{1}{1+x^2+\cosh(x)(2+\cosh(x))}\big(2x,1-x^2-\cosh^2(x)\big)^t.
    \end{equation}
    Note that any isometry $\D^2\to\D^2$ is given by a Möbius transformation keeping $\D^2$ fixed. More precisely, using complex notation for $\C=\R^2$, all isometries are given by
    \begin{equation}
        F_{\theta,c}\colon\D^2\to\D^2,\quad z\mapsto e^{i\theta} \frac{z-c}{\overline{c}z-1}
    \end{equation}
    where $\theta\in\R$ and $c\in \D^2$ are arbitrary. Since rotations are isometries of $\R^2$ and thus preserve Euclidean lengths of curves, one can restrict to $\theta=0$. Then, for any $c\in\D^2$, one needs to verify that $|\partial_x (F_{0,c}\circ\gamma)|\in L^1(\R)$ which is a straight-forward computation.

    Indeed, viewing $\D^2$ as a subset of $\C$, one finds that, for $\gamma_c\vcentcolon= F_{0,c}\circ\gamma$, 
    \begin{equation}
        \partial_x\gamma_c(x)=\frac{2 \left(| c| ^2-1\right) (1-i \sinh (x))}{\left(i \overline{c} (i
        x+\cosh (x)-1)+i x+\cosh (x)+1\right)^2}.
    \end{equation}
    Note here that $i\overline{c}+1$ is never equal to $0$ since $c\in\D^2$, i.e. $|c|<1$. So the dominating term in the denominator for $|x|\to\infty$ of $|\partial_x\gamma_c|$ is $|i\overline{c}+1| \cosh^2(x)$ while the numerator can be estimated by $4(1+|\sinh(x)|)$. This clearly gives integrability on $\R$. So also all asymptotically geodesic elastica in $\D^2$ have finite Euclidean length.
\end{proof}

\section*{Acknowledgments}
The author would like to thank Anna Dall’Acqua for helpful discussions and comments. Moreover, the author is grateful to the anonymous referee for their valuable comments on the original manuscript.
\section*{Data availability} 
Data sharing not applicable to this article as no datasets were generated or analyzed during the current study. 

\newcommand{\etalchar}[1]{$^{#1}$}

\end{document}